\documentclass{amsart}
\usepackage[utf8]{inputenc}

\usepackage{epsdice}

\usepackage{color}

\usepackage[OT2,T1]{fontenc}

\usepackage[final]{microtype}

\usepackage{tikz}
\usepackage{tikz-cd}
\usepackage{xypic}

\usepackage[draft=false,linktocpage=true]{hyperref}
\usepackage[capitalise]{cleveref}

\usepackage{mathtools}
\usepackage{amsmath}

\usepackage{amsthm,amssymb}
\numberwithin{equation}{section}

\theoremstyle{plain}
\newtheorem{theorem}[equation]{Theorem}

\newtheorem{proposition}[equation]{Proposition}
\newtheorem{lemma}[equation]{Lemma}

\newtheorem{corollary}[equation]{Corollary}

\theoremstyle{definition}
\newtheorem{definition}[equation]{Definition}

\newtheorem{example}[equation]{Example}

\newtheorem{remark}[equation]{Remark}

\newtheorem{warning}[equation]{Warning}

\title{logarithmic de Rham comparison for open rigid spaces}
\author{Shizhang Li}
\address{Department of Mathematics, Columbia University, MC 4406, 2990 Broadway,
New York, NY, 10027, U.S.}
\email{shanbei@math.columbia.edu}

\author{Xuanyu Pan}
\address{Institute of Mathematics, AMSS, Chinese Academy of Sciences, 55 ZhongGuanCun East Road, Beijing, 100190, China and Max Planck Institute for Mathematics, Vivatsgasse 7, Bonn, Germany 53111}
\email{pan@amss.ac.cn}
\email{panxuanyu@mpim-bonn.mpg.de}
\date{\today}

\thanks{${}^{\dagger}$ The second named author is partially supported by NSFC No. 11688101.}

\begin{document}

\maketitle

\begin{abstract}
  In this note, we prove the logarithmic $p$-adic comparison theorem for open
  rigid analytic varieties. We prove that a smooth rigid analytic variety with a
  strict simple normal crossing divisor is locally $K(\pi,1)$ (in a certain
  sense) with respect to $\mathbb{F}_p$-local systems and ramified coverings
  along the divisor. We follow Scholze's method to produce a pro-version of the
  Faltings site and use this site to prove a primitive comparison theorem in our
  setting. After introducing period sheaves in our setting, we prove aforesaid
  comparison theorem.
\end{abstract}

\tableofcontents

\section{Introduction}

Historically, classical Hodge theory was developed from Hodge's results up
through Deligne's papers on mixed Hodge structures in the early 1970's. The
famous decomposition theorem is the following.
\begin{theorem}[Hodge, Deligne]
Let $X$ be a smooth proper variety over complex numbers $\mathbb{C}$ 
with a strict simple normal crossing divisor $D$. Then we have
\[
  H^m_{sing}(X-D,\mathbb{Z})\otimes_{\mathbb{Z}} \mathbb{C} \cong
  H^m(X,\Omega_X^{\bullet}(\log D)) \cong \bigoplus\limits_{i+j=m}
  H^i(X,\Omega_X^j(\log D)),
\]
where $\Omega_{X}^j(\log D)$ is the sheaf of $j$-forms with logarithmic
singularities along $D$ on $X$.
\end{theorem}

The $p$-adic Hodge theory properly began around 1966 when Tate~\cite[p.~180
Remark]{p-divisible} proved a $p$-adic version of the comparison theorem for an
abelian variety of good reduction over a $p$-adic field. After the works of
Fontaine, Messing, Bloch, Kato, et al., Faltings proved the following.

\begin{theorem}
  \label{faltings} \cite{Fal} Let $X$ be a smooth proper variety over a $p$-adic
  field $k$ with a strict simple normal crossing divisor $D$. Then there exists
  a $\mathrm{Gal}(\overline{k}/k)$-equivariant isomorphism
\[
  H^m_{\textnormal{\'et}}((X-D)_{\overline{K}},\mathbb{Q}_p)\otimes_{\mathbb{Q}_p}
  \mathbb{C}_{p} \cong \bigoplus\limits_{i+j=m} H^i(X,\Omega_X^j(\log
  D))\otimes_K \mathbb{C}_p(-j),
\]
where $\Omega_{X}^j(\log D)$ is the sheaf of $j$-forms with logarithmic
singularities along $D$ on $X$ and $\mathbb{C}_p=\hat{\overline{\mathbb{Q}}}_p$.
\end{theorem}

Afterwards, many people have found other ways to produce this comparison isomorphism.
There is another remarkable approach to prove such a comparison, due to
Beilinson (see~\cite{Beilinson}), using derived de Rham cohomology (of
Illusie), $h$-topology and de Jong’s alterations.

Recently, Scholze generalized Theorem~\ref{faltings}, namely,
the de Rham comparison theorem for a smooth proper rigid analytic space over a
$p$-adic field. Moreover, the comparison theorem that he proved allows
coefficients to be local systems, see~\cite[Theorem 8.4]{Sch13}. However, his
theorem does not include the logarithmic case. The purpose of this note is to
prove the de Rham comparison in the logarithmic case (for constant coefficients)
using the same methods.

It is also worth mentioning that in the work of Colmez--Nizio{\l}, they proved a
semistable comparison for semistable formal log-schemes~(see~\cite[Corollary
5.26]{C-N}). In particular, they've already obtained the de Rham comparison in
the logarithmic case (for constant coefficients) assuming the appearance of a
semistable formal model.


Let $k$ be a discretely valued complete non-archimedean extension of
$\mathbb{Q}_p$ with perfect residue field $\kappa$. Our main comparison theorem
(see~Theorem~\ref{first comparison} and Theorem~\ref{second comparison}) is the
following:

\begin{theorem}
\label{Main Theorem}
Let $X$ be a proper smooth adic space over $\mathrm{Spa}(k,\mathcal{O}_k)$ 
with a strict simple normal crossing divisor $D$ and 
complement $U \coloneqq X \setminus D$. 
Then, there is a natural $\mathrm{Gal}(\overline{k}/k)$-equivariant isomorphism
\[
H^i_{\mathrm{\acute{e}t}}(U_{\bar{k}},\mathbb{Z}_p) 
\otimes_{\mathbb{Z}_p} B_{\mathrm{dR}} \cong
H^i(X,\Omega^{\bullet}_{X}(\log D)) \otimes_k B_{\mathrm{dR}}
\]
preserving filtrations. Moreover, the logarithmic Hodge--de Rham spectral sequence
\[
  \xymatrix{ E_1^{i,j} = H^j(X,\Omega_X^i(\log D))\ar@{=>}[r] &
    H^{i+j}(X,\Omega^{\bullet}_{X}(\log D)) }
\]
degenerates. In particular, the logarithmic Hodge--Tate spectral sequence also degenerates and yields the logarithmic Hodge--Tate decomposition
\[
H^i(U_{\overline{k}},\mathbb{Q}_p)\otimes_{\mathbb{Q}_p} \hat{\bar{k}} \cong
\bigoplus\limits_j H^{i-j}(X,\Omega_X^j(\log D ))\otimes_k \hat{\bar{k}}(-j). 
\]
\end{theorem}

During the preparation of this note, we learned that Hansheng Diao, Kai-Wen Lan,
Ruochuan Liu and Xinwen Zhu have proved a more
powerful version of this comparison theorem including allowing more general
coefficients (see~\cite[Theorem 1.1]{DLLZ}).

We hope our approach following Scholze and using the Faltings site is still
interesting in its own. 

In the rest of this introduction, we give a brief
descriptions of the organization of this note.
In Subsection~\ref{subsection 2.3}
we introduce the Faltings site $X_{\log}$ and show that the complement of a
strict simple normal crossing divisor is locally $K(\pi,1)$ (in a certain sense)
with respect to $\mathbb{F}_p$-local systems, see~\Cref{comparison logtopos}
and~\Cref{corokill}.\footnote{Shortly after posting the first version of this
  note on arXiv, we are pointed out that similar result has been obtained by
  Colmez--Nizio\l~(see~\cite[5.1.4]{C-N}) via similar methods.} The main
consequence is that we can compute the cohomology of local systems on
$U_{\textnormal{\'{e}t}}$ via $X_{\log}$. The main ingredients for the proofs
are results of L\"{u}tkebohmert in~\cite{Lut93}, Scholze's $K(\pi,1)$-result for
affinoid spaces and Gysin sequence.

In \Cref{section 3}, we introduce a general method to produce a pro-site $X_{\mathrm{prolog}}$ 
of the Faltings site $X_{\log}$. This general method is recapturing~\cite[Section 3]{Sch13}.
We also show that the pro-site $X_{\mathrm{prolog}}$ shares a lot of good properties, 
e.g.~algebraicity and it has a coherent terminal object if the rigid space $X$ is proper over $k$.
Most of the arguments are formal and similar to counterparts in~\cite[Section 3]{Sch13}.

In \Cref{section 4}, we introduce structure sheaves on $X_{\log}$ and $X_{\mathrm{prolog}}$.
We also show that $X_{\mathrm{prolog}}$ has affinoid perfectoid basis, see \Cref{lemmascholze410}.
The main difference of $X_{\mathrm{prolog}}$ from the pro-\'{e}tale site
is that we are allowed to take any root
(not just $p$-root) of the coordinates defining the divisor D,
see \Cref{model example}. This difference is clear from \cite{Fal}. 

In \Cref{section 5}, we follow the method of Scholze to show the primitive comparison 
Theorem \ref{thmpc} in our setting. A similar result has been
obtained by Diao in the setting of
(pro)-Kummer \'{e}tale site, see~\cite[Proposition 4.4]{log-adic}. To show the comparison theorem for $X_{\mathrm{prolog}}$, 
we need to enhance some Scholze's results in the case allowing ramified coverings.

In \Cref{section 6}, we introduce the period sheaves on $X_{\mathrm{prolog}}$. 
The new ingredient is a logarithmic version of the period sheaf, $\mathcal{O}\mathbb{B}^+_{\mathrm{logdR}}$. 
The main result of this section is the logarithmic Poincar\'{e} Lemma, see Corollary \ref{corologPL}.

In Subsection~\ref{subsection 7.1}, we introduce a notion of vector bundles on the Faltings site $X_{\log}$
and prove Theorem A and Theorem B \`{a} la Cartan for them. 
Then in Subsection~\ref{subsection 7.2} we prove the aforesaid de Rham comparison theorem.

\subsection*{Acknowledgments}

The first author wants to express his gratitude to his advisor 
Johan de Jong for his warm encouragement
and all the discussions on this project. The first author also wants to thank David Hansen
for lecturing on~\cite{Sch13} at Columbia University, which inspired him to work on this project 
as an initiative to understand~\cite{Sch13} better, as well as for keeping track of his progress 
and sharing the excitements.
The second author is thankful to Gerd Faltings for his two lectures
on $p$-adic Hodge theory at the Max Planck Institute
which brings his attention to this project.
We thank Bhargav Bhatt, Dingxin Zhang and Weizhe Zheng for helpful discussions.

We thank Hansheng Diao, Kai-Wen Lan,
Ruochuan Liu and
Xinwen Zhu for communicating with us. We especially thank
Xinwen Zhu heartily for many precious comments and advice on an early draft of this note.

We also thank Pierre Colmez and Wies{\l}awa Nizio{\l} for communications concerning the first version of this note.

\subsection*{Notations and Conventions}

In this note, unless specified otherwise, we will use the following notations
and conventions. Let \(k\) be a \(p\)-adic field, i.e., discretely valued
complete non-archimedean field extension of \(\mathbb{Q}_p\) with perfect
residue field. We denote its ring of integers by \(\mathcal{O}_k\). We will use
$K$ to denote a perfectoid field which is the completion of some algebraic
extension of $k$.

Let \(X\) be a smooth proper rigid space over \(k\) of dimension \(n\) and let
\(D = \bigcup_{i \in I} D_i \subset X\) be a divisor, here \(I\) is a finite
index set. For any subset \(J \subset I\), we use \(D^J\) to denote \(\bigcap_{i
  \in J} D_i\). We say \(D\) is a strict simple normal crossing (shorthand by
SSNC from now on) divisor if all of \(D^{J}\)'s are smooth of codimension
\(\lvert J \rvert\) where \(\lvert J \rvert \coloneqq \text{number of elements
  in } J\). Here $D^J$ has codimension greater than $n$ means it is empty. We denote \(X
\setminus D\) by \(U\). For any rigid space \(V \to X\) admitting a map to
\(X\), we denote the preimage of \(U\) by \(V^{\circ}\).

We use notation \(\mathbb{D}^r(\underline{T})\) to denote \(r\)-dimensional
unit polydisc with coordinates given by \(T_i\). We denote 
$\mathbb{D}^r(\underline{T}) \setminus V(T_1 \cdots T^r)$ by 
$\mathbb{D}^{\circ,r}(\underline{T})$.

Let \(A\) be a ring, we denote its normalization by \(A^{\nu}\). If $f_1,\ldots,f_r$ are $r$ elements in
an affinoid algebra $A$, then we denote $\mathrm{Sp}((A[\sqrt[m]{f_1},\ldots,\sqrt[m]{f_r}])^\nu)$ by
$\mathrm{Sp}(\mathrm{A}[\sqrt[m]{f_l}])$.

We use both the language of adic spaces of finite type over $\mathrm{Spa}(k,\mathcal{O}_k)$ and rigid
spaces over $\mathrm{Sp}(k)$ interchangeably, we hope this does not confuse the reader.

\section{Preliminaries}
\label{section 2}

\subsection{Abhyankar's Lemma}
\label{subsection 2.1}

Let us discuss Abhyankar's Lemma for rigid spaces over \(p\)-adic fields. This
is more or less already obtained by L\"{u}tkebohmert in~\cite{Lut93}, see
also~\cite[Section 2.2]{Hansen17}.

\begin{proposition}
\label{urgent Abhyankar}
Let \(S\) be a smooth rigid space over \(k\) which is not necessarily
quasi-compact or quasi-separated. Let
\[
  \phi \colon Y \to X = S \times \mathbb{D}^{\circ,r}(\underline{z})
\]
be a finite \'{e}tale covering of degree \(d\). Then after pulling \(Y\) back to
\(S \times \mathbb{D}^{\circ,r}(\underline{T})\) along 
\[
S \times \mathbb{D}^{\circ,r}(\underline{T}) \to S \times \mathbb{D}^{\circ,r}(\underline{z}),~z_i \mapsto T_i^{d!}
\] 
it extends to a finite \'{e}tale covering of \(S \times
\mathbb{D}^r(\underline{T})\).
\end{proposition}

\begin{proof}
  Step 1: let us first prove this in the case where \(r=1\). Since extension of
  covering is faithful (see~\cite[Proposition 2.9]{Hansen17}), by descent it
  suffices to prove the statement after replacing \(S\) by an \'{e}tale cover of
  \(S\). Therefore we may assume that the conditions of~\cite[Lemma 3.2]{Lut93}
  are satisfied. Our statement just follows from~\cite[Lemma 3.2]{Lut93}.

  Step 2: let us prove the general case by induction on \(r\). Write 
  \[S \times \mathbb{D}^{\circ,r} = S \times (\mathbb{D}^{\circ,r-1}) \times (\mathbb{D} \backslash \{0\}).\] 
  By step 1 we see that after pulling \(Y\) back along 
  \[
  S \times (\mathbb{D}^{\circ,r-1}(\underline{z})) \times (\mathbb{D}(T_n) \backslash \{0\}) \to S \times (\mathbb{D}^{\circ,r-1}(\underline{z})) \times (\mathbb{D}(z_n) \backslash \{0\}),~z_n \mapsto T_n^{d!}
  \]
  it extends to a finite \'{e}tale
  covering of \(S \times (\mathbb{D}^{\circ,r-1}(\underline{z})) \times \mathbb{D}(T_n)\). Now by induction, we are done.
\end{proof}

From the proposition above, we can deduce the following Theorem which can be
thought of as the analogue of Abhyankar's Lemma in rigid geometry.

\begin{theorem}[Rigid Abhyankar's Lemma]
\label{Abhyankar's Lemma}
  Let \(N_1 = \mathrm{Sp}(A)\) be a smooth affinoid space over a \(p\)-adic
  field, let \(f_i\) be \(r\) functions which cut out \(r\) smooth divisors.
  Denote the union of these divisors by \(D\). Let \(N_2 \to N_1\) be a finite
  morphism which is \'{e}tale away from \(D\) with \(N_2 = \mathrm{Sp}(B)\)
  being normal. Then for sufficiently divisible \(k \in \mathbb{N}\), the map
  \[
    A[\sqrt[k]{f_i}] \to (B \otimes_A A[\sqrt[k]{f_i}])^{\nu}
  \]
  is finite \'{e}tale.
\end{theorem}

\begin{proof}
  Since the statement is local on \(N_1\), we may assume (by~\cite[Theorem
  1.18]{Kiehl67}, see also~\cite[Theorem 2.11]{Mitsui}) that \(A = A_0 \langle
  T_i \rangle\) where the divisor \(D\) is cut out by \(T_1T_2 \cdots T_r\).
  Then we see that \(f_i = g_i \cdot T_i\) with \(g_i\) units.

For a \(k \in \mathbb{N}\) to be chosen later, we let \(X = \mathrm{Sp}(A
\langle \sqrt[k]{f_i} \rangle)\), \(Y = \mathrm{Sp}(A \langle \sqrt[k]{T_i}
\rangle)\) and \(W = (X \times_{N_1} Y)^{\nu}\). Note that \(W \to X\) and \(W
\to Y\) are both finite \'{e}tale, since they are given by adjoining \(k\)-th
root of \(g_i\).

\[
  \xymatrix{
    X \ar[d] & W \ar[l] \ar[d] \\
    N_1 & Y \ar[l]
  }
\]

What we need to show is that after choosing \(k\) sufficiently divisible, the
base change map
\[
  (N_2 \times_{N_1} X)^{\nu} \to X
\]
is finite \'{e}tale. But since \(W \to X\) is finite \'{e}tale, by descent it is
enough to check after base changing to \(W\). Because \(W \to N_1\) also factors
through \(Y\), it suffices to choose \(k\) so that the base change to \(Y\) is
\'{e}tale. This can be achieved by~\Cref{urgent Abhyankar}.
\end{proof}

\subsection{Gysin sequence}
\label{subsection 2.2}

Let us gather facts concerning Gysin sequence (cohomological purity) in the
setup of rigid spaces as developed by Berkovich (see~\cite{Berk95}) and Huber
(c.f.~\cite[Section 3.9]{Huber96}).

\begin{theorem}[Gysin sequence]
\label{Gysin}
Let \(Y\) be a smooth rigid space over \(k\), \(\mathcal{E}\) an
\(\mathbb{F}_p\)-local system on \(Y\) and \(Z \subset Y\) a smooth divisor on
\(Y\). Then we have a long exact sequence
\[
  H^2_{\mathrm{\acute{e}t}}(Y,\mathcal{E}) \to H^2_{\mathrm{\acute{e}t}}(Y \setminus Z,
  \mathcal{E}|_{Y \setminus Z}) \to H^1_{\mathrm{\acute{e}t}}(Z, \mathcal{E}|_Z(-1)) \to
  H^3_{\mathrm{\acute{e}t}}(Y,\mathcal{E}) \to \ldots.
\]
Here $\mathcal{E}|_Z(-1))$ means the Tate twist of the pullback $\mathcal{E}|_Z$ of $\mathcal{E}$ to $Z$.
\end{theorem}

\begin{proof}
  This follows from a re-interpretation of~\cite[2.1 Theorem]{Berk95}, where we apply the Theorem in loc.~cit.~to our case where $S$ and $(Y,X)$ from loc.~cit.~correspond to $\mathrm{Sp}(k)$ and $(Z,Y)$, which satisfies the condition in loc.~cit.
\end{proof}

\subsection{The site $X_{\log}$}
\label{subsection 2.3}

In this subsection we introduce the log-\'{e}tale site \(X_{\log}\) (also known as the
Faltings site)
of the pair \((X,D)\) and show a comparison theorem between this site and \(U_{\mathrm{\acute{e}t}}\).
Note that this site depends on a choice of divisor \(D\), however we suppress that
in the notation for the sake of simplicity of notations.

\begin{definition}
\label{logtopos}
Let \(f\) be a morphism between two objects \(V_i\rightarrow X\) over \(X\) for \(i = 1, 2\). 
We denote the restriction of \(f\) to \(V_2^{\circ}\) by $f^{\circ}$.

We define a site \(X_{\log}\) as follows: an object of \(X_{\log}\) consists of
arrows
\[
  N \xrightarrow{f} V \xrightarrow{g} X
\]
(denoted by $(V,N)$) such that 
\begin{enumerate}
\label{condition}
\item the morphism $g$ is \'{e}tale;
\item $N$ is normal;
\item the morphism $f$ is finite with $f^{\circ} \colon N \backslash (g \circ f)^{-1}(D) \to V \backslash g^{-1}(D)$ being \'{e}tale and;
\item $(g \circ f)^{-1}(D)$ is nowhere dense in $N$.
\end{enumerate}
A morphism in this site from $(V,N)$ to
$(V',N')$ is given by a pair $(p,q)$ of two $X$-maps in a commutative diagram:
\[
  \xymatrix{
    N \ar[r]^q \ar[d] & N' \ar[d] \\
    V \ar[r]^p        & V'.
  }
\]

The morphisms 
\[
\{(p_i,q_i):(V_i,N_i) \rightarrow (V,N)\}
\]
form a covering if $N=\bigcup q_i(N_i)$. 
Notice that by~\Cref{lemmopen0} (2) below, the image of \(N_i\) in \(N\) are open subsets.

Similarly, for a $V \to X$ \'{e}tale over $X$ we can define a subsite $V_{\mathrm{f,log}}$
whose objects are consisting of $N \xrightarrow{f} V$ satisfying condition~\ref{condition}(2)-(4).
The morphisms are just usual morphisms in the category of rigid spaces over $V$.
Note that by~\cite[Theorem 1.6]{Hansen17}, we have $V_{\mathrm{f,log}} \cong V^{\circ}_{\mathrm{f\acute{e}t}}$.
\end{definition}

\begin{remark}
One should note the subtle difference between the above definition of the Faltings site
and that in~\cite[III.8.2]{AGT}. In particular, the counterpart of the counterexample
in~\cite[III.8.18]{AGT} in the Faltings site here does not form a covering.
\end{remark}

Before introducing the following Lemma, let us fix some notation. 
Given $(V',N') \rightarrow (V,N)$ and $(V'',N'') \rightarrow (V,N)$ in \(X_{\log}\),
let $W \coloneqq V' \times_V V''$, and observe that
$N' \times_N N'' = (N' \times_{V'} W) \times_{(N \times_V W)} (N'' \times_{V''} W)$.

\begin{lemma}
\label{lemmopen0}
The category $X_{\log}$ has the following properties:
\begin{enumerate}
\item finite projective limit and a terminal object exist. Moreover, the fiber products of 
$(V',N') \rightarrow (V,N)$ and $(V'',N'') \rightarrow (V,N)$ is given by 
\((W, (N' \times_N N'')^{\nu})\) (following the notation prior to this Lemma).
 In particular, the equalizer of two morphisms 
$(p,q), (s,t): (V',N') \rightarrow (V,N)$ is given by $\bigg(eq(p,s), eq(q,t)^{\nu}\bigg)$ where $eq(\cdot,\cdot)$ is the equalizer of the two morphisms;
\item the image of the morphism $(V,N)\rightarrow (V',N')$ in $|N'|$ is open
  and;
\item $(V,N)$ is quasi-compact (resp.~quasi-separated) if and only if $N$ is
  quasi-compact (resp.~quasi-separated) which will be valid if $V$ is quasi-compact (resp.~quasi-separated).
\end{enumerate}
\end{lemma}

Before giving the proof of this Lemma, we remark that this Lemma can be plainly
generalized to analogous statements for $X_{\mathrm{f,log}}$.
But since we do not need it in the rest of this article, we do not state them here.

\begin{proof}
  Proof of (1). The existence of finite projective limit and the explicit descriptions just follow
  from~\cite[Theorem 1.6]{Hansen17} and the existence and descriptions in
  \(X_{\mathrm{\acute{e}t}}\) (for the \(V\) part) and \(V^{\circ}_{f {\mathrm{\acute{e}t}}}\)
  (for the \(N\) part). The terminal object is clearly \((X,X)\). 

  Proof of (2). Let us consider the morphism $N \to N'_V \coloneqq N'
  \times_{V'} V$. We claim that the image is union of connected components of
  $N'_V$, (2) clearly follows from this claim. This claim follows from the fact
  that ${N'_V}^{\circ}$ has the same number of connected components as that of
  $N'_V$ (see~\cite[Corollary 2.7]{Hansen17}) and is dense in $N'_V$. But now
  since $N^{\circ} \to {N'_V}^{\circ}$ is finite \'{e}tale, therefore the image
  is union of connected components of ${N'_V}^{\circ}$. Because $N \to N'_V$ is
  finite, therefore the image is closure of the image of corresponding circ map.

  Proof of (3). Let us first show that if $N$ is quasi-compact, then $(V,N)$ is a quasi-compact
  object in this site. Let $(V_i,N_i) \to (V,N)$ be a covering. Because the
  image of $N_i$ is a union of connected components of preimage of image of
  $V_i$, we see that it must be an open subset of $N$. Since $N$ is
  quasi-compact, finitely many of $N_i \to N$ would have
  image covering $N$. Now if $(V,N)$ is a quasi-compact object, it
  is obvious that the image of $N$ in $V$ is quasi-compact. Hence $N$ being finite over that image, is also quasi-compact. The statement concerning
  quasi-separatedness just follows from the description of fibre product. The statement about $V$ easily follows from the fact that $N \to V$ is finite.
\end{proof}

\begin{definition}
\label{defux}
There is a natural morphism between sites
\[
  U_{\mathrm{\acute{e}t}} \rightarrow X_{\log},~(V,N) \mapsto N^{\circ}
\]
inducing a morphism between topoi
\[
  u_X:\mathrm{Sh} (U_{\mathrm{\acute{e}t}}) \rightarrow \mathrm{Sh} (X_{\log}).
\]
\end{definition}

The main result of this section is the following.

\begin{theorem}
\label{comparison logtopos}
Let \(\mathbb{L}\) be a $\mathbb{F}_p$-local system on \(U_{\mathrm{\acute{e}t}}\). 
Then we have
\begin{enumerate}
\item $u_{X*}(\mathbb{L})(V,N)=\mathbb{L}(N^{\circ})$ for an object $(V,N)\in
  X_{\log}$ and;
\item $R^iu_{X*}(\mathbb{L})=0$ for $i \geq 1$.
\end{enumerate}
\end{theorem}

Before proving this Theorem, let us state and prove some Lemmas.

\begin{lemma}
\label{lemmext}
Let $\mathcal{E}$ be a $\mathbb{F}_p$-local system on \(S \times
\mathbb{D}^{\circ,r}\) where \(S\) is a smooth connected affinoid space over
\(k\). Then there is a Kummer map \(\mathbb{D}^r \xrightarrow{\varphi}
\mathbb{D}^r\) (i.e.~raise coordinates to sufficiently divisible power) such
that $(id_S\times \varphi)^*(\mathcal{E})$ is a restriction of a
$\mathbb{F}_p$-local system on $S \times \mathbb{D}^r$.
\end{lemma}

\begin{proof}
  It follows from~\Cref{urgent Abhyankar} and the fact that $\mathcal{E}$ is
  represented by a finite \'{e}tale covering of $S \times \mathbb{D}^{\circ,r}$.
\end{proof}

\begin{lemma}
\label{lemmakill}
Let $\mathcal{E}$ be a $\mathbb{F}_p$-local system on \(S \times \mathbb{D}^r
\times \mathbb{D}^k = S \times \mathbb{D}^{r+k}\) where \(S\) is a smooth
connected affinoid space over \(k\). Then for every cohomology class $\alpha \in
H^j_{\mathrm{\acute{e}t}}(S \times \mathbb{D}^{\circ,r} \times \mathbb{D}^k,
\mathcal{E})$ where \(j \geq 1\), there is a finite \'{e}tale covering
$N^{\circ} \xrightarrow{\varPhi} S \times \mathbb{D}^{\circ,r} \times
\mathbb{D}^k$ with $\varPhi^*(\alpha)=0$ in $H^j_{\mathrm{\acute{e}t}}
(N^{\circ},\varPhi^* \mathcal{E})$.
\end{lemma}

\begin{proof}
  For $j=1$, the lemma is easily deduced from the torsor interpretation of
  cohomology classes of degree $1$.

  In the following, we assume $j$ is at least $2$. We prove the lemma by the
  induction on $r$. When $r=0$, it is a special case of~\cite[Theorem
  4.9]{Sch13}. Suppose that the lemma holds for $r$. We consider the
  $(r+1)$-case.

  Note that
  \[
    S \times \mathbb{D}^{\circ,r+1} \times \mathbb{D}^k = (S \times
    \mathbb{D}^{\circ,r} \times \mathbb{D}^{k+1}) \setminus (S \times
    \mathbb{D}^{\circ,r} \times \mathbb{D}^k \times \{0\} ) \eqqcolon C
    \setminus \Delta_1
  \]
  where $C = S \times \mathbb{D}^{\circ,r} \times \mathbb{D}^{k+1}$ and
  $\Delta_1 = S \times \mathbb{D}^{\circ,r} \times \mathbb{D}^k \times \{0\}$.
  The Gysin sequence (\Cref{Gysin}) applied to the pair $(C, \Delta_1)$ gives
  the connecting map
  \[
    H^j_{\mathrm{\acute{e}t}}(C \setminus \Delta_1,\mathcal{E}) \rightarrow
    H^{j-1}_{\mathrm{\acute{e}t}}(\Delta_1,\mathcal{E}|_{\Delta_1}(-1))
  \]
  mapping $\alpha$ to $\beta$. By the induction and $j\geq 2$, there is a finite
  \'{e}tale covering $f \colon \Delta_1' \rightarrow \Delta_1$ with
  $f^*(\beta)=0$. It give us a finite covering of $C$
  \[
    f \times id_{\mathbb{D}}: \Delta_1' \times \mathbb{D} \rightarrow \Delta_1
    \times \mathbb{D} = S \times \mathbb{D}^{\circ,r} \times \mathbb{D}^k \times
    \mathbb{D} = C
  \]
  whose restriction to $\Delta_1 \times \{0\}$ is $f$. The map $f \times
  id_{\mathbb{D}}$ induces a map from the Gysin sequence of $(\Delta_1' \times
  \mathbb{D}, \Delta_1')$ to that of $(C, \Delta_1)$ as follows:
  \[
    \xymatrix{ H^j_{\mathrm{\acute{e}t}}(C) \ar[r] \ar[d]_{(f \times id_{\mathbb{D}})^*}
      & H^j_{\mathrm{\acute{e}t}}(C \setminus \Delta_1) \ar[r] \ar[d]_{(f \times
        id_{\mathbb{D}})|_{C \setminus \Delta_1}^*} &
      H^{j-1}_{\mathrm{\acute{e}t}}(\Delta_1) \ar[d]_{f^*} \ar[r] &
      H^{j+1}_{\mathrm{\acute{e}t}}(C) \ar[d]_{(f \times id_{\mathbb{D}})^*} \\
      H^j_{\mathrm{\acute{e}t}}(\Delta_1' \times \mathbb{D}) \ar[r]^h &
      H^j_{\mathrm{\acute{e}t}}(\Delta_1' \times \mathbb{D} \setminus \Delta_1') \ar[r] &
      H^{j-1}_{\mathrm{\acute{e}t}}(\Delta_1') \ar[r]& H^{j+1}_{\mathrm{\acute{e}t}}(C)}.
  \]
  It follows that $(f \times id_{\mathbb{D}})|_{C \setminus
    \Delta_1}^*(\alpha)=h(\gamma)$ for some $\gamma \in
  H^j_{\mathrm{\acute{e}t}}(\Delta_1' \times \mathbb{D}, \mathcal{E}|_{\Delta_1' \times
    \mathbb{D}})$.

  We claim that there is a finite \'{e}tale covering $\theta \colon N
  \rightarrow \Delta_1' \times \mathbb{D}$ with $\theta^*(\gamma)=0$. This
  proves that
  \[
    \left(\theta \circ (f \times id_{\mathbb{D}})|_{C \setminus
        \Delta_1}\right)^*(\alpha)=0
  \]
  which is what we need to show in the $(r+1)$-case.

  Now we show the claim above. In fact, let $\mathcal{E}'$ be the
  \(\mathbb{F}_p\)-local system \((f \times
  id_{\mathbb{D}})_*(\mathcal{E}|_{\Delta'_1 \times \mathbb{D}})\) on $C$. Now
  $\gamma$ can be viewed as an element in $H^j_{\mathrm{\acute{e}t}}(C,\mathcal{E}') =
  H^j_{\mathrm{\acute{e}t}}(\Delta_1' \times \mathbb{D}, \mathcal{E}|_{\Delta_1' \times
    \mathbb{D}})$. By \Cref{lemmext}, there is a finite \'{e}tale
  covering
  \[
    \varphi \colon S \times \mathbb{D}^{\circ,r} \times \mathbb{D}^{k+1}
    \rightarrow S \times \mathbb{D}^{\circ,r} \times \mathbb{D}^{k+1} =
    \Delta_1\times \mathbb{D} = C
  \]
  such that the pullback $\varphi^*(\mathcal{E}')$ is a restriction of a
  $\mathbb{F}_p$-local system on $S \times \mathbb{D}^r \times
  \mathbb{D}^{k+1}$. Therefore by the induction (applied to $\varphi^*(\gamma)$), we have a finite 
  \'{e}tale covering $ g \colon W \rightarrow S\times \mathbb{D}^{\circ,r} \times \mathbb{D}^{k+1}$ with 
  $(\varphi \circ g)^*(\gamma)=g^*(\varphi^*(\gamma))=0.$ Considering the Cartesian diagram,
\[ 
\xymatrix{
N \ar[d] \ar[r]^\theta & \Delta'_1 \times \mathbb{D} \ar[d]^{f \times id_{\mathbb{D}}} \\
W \ar[r]^{\varphi \circ g} & \Delta \times \mathbb{D}=C
} 
\]
we see that $\theta$ is a finite \'{e}tale covering with $\theta^*(\gamma)=0$.
\end{proof}

\begin{lemma}
\label{propkill}
Let $\mathcal{E}$ be a $\mathbb{F}_p$-local system on $S \times
\mathbb{D}^{\circ,r} \times \mathbb{D}^k$ where $S$ is a smooth connected
affinoid space over \(k\). Then for every cohomology class $\alpha \in
H^j_{\mathrm{\acute{e}t}}(S \times \mathbb{D}^{\circ,r} \times \mathbb{D}^k,
\mathcal{E})$ where $j\geq 1$, there is a finite \'{e}tale covering $N^{\circ}
\xrightarrow{\varphi} S \times \mathbb{D}^{\circ,r} \times \mathbb{D}^k$
with
\[
  \varphi^*(\alpha)=0~\text{ in }~H^j_{\mathrm{\acute{e}t}}(N^{\circ},\varphi^*
  \mathcal{E}).
\]
\end{lemma}

\begin{proof}
  It follows from \Cref{lemmext} and \Cref{lemmakill}.
\end{proof}

\begin{proposition}
\label{corokill}
Let $S$ be a smooth connected affinoid space over $K$, and let $\mathbb{D}^r$ be
the unit ball with coordinates $z_1,\ldots,z_r$. Set $\Delta=V(z_1\cdots z_r)$
and $\mathbb{D}^{\circ,r}=\mathbb{D}^r-\Delta$. Let $f:N^{\circ} \rightarrow S
\times \mathbb{D}^{\circ,r}$ be a finite \'{e}tale covering. For a
$\mathbb{F}_p$-local system $\mathbb{L}$ on $S \times \mathbb{D}^{\circ,r}$ and
a cohomology class $\alpha \in H^i_{\mathrm{\acute{e}t}}(N^{\circ},
\mathbb{L}|_{N^{\circ}})$ where $i\geq 1$, there is a finite \'{e}tale covering
$\varphi:M^{\circ} \rightarrow N^{\circ}$ such that
\[
  \varphi^*(\alpha) = 0 \in H^i_{\mathrm{\acute{e}t}}(M^{\circ},\mathbb{L}|_{M^{\circ}}).
\]
\end{proposition}

\begin{proof}
  This follows from applying \Cref{propkill} to $\mathcal{E}=f_*(\mathbb{L})$
  and \(f_*(\alpha)\).
\end{proof}

Now we are ready to give the

\begin{proof}[Proof of \Cref{comparison logtopos}]
  The statement (1) is obvious. It is clear that $R^iu_{X*}(\mathbb{L})$ is the
  sheaf associated to the presheaf
  \[
    (N \xrightarrow{f} V \xrightarrow{g} X) \rightarrow
    H^i_{\mathrm{\acute{e}t}}(N^{\circ},(g\circ f)^{*}\mathbb{L})=
    H^i_{\mathrm{\acute{e}t}}(N^{\circ},\mathbb{L}|_{N^{\circ}}).
  \]
  The statement (2) is a local property, hence (by~\cite[Theorem 1.18]{Kiehl67})
  we may assume that $V$ is $S \times \mathbb{D}^r$ with finite
  \[
    f: N \rightarrow S \times \mathbb{D}^r =V
  \]
  such that $f^{\circ}$ is \'{e}tale and $V^{\circ} = S \times \mathbb{D}^{\circ,r}$ where $S$
  is a smooth and connected affinoid space over \(k\). By~\cite[Theorem
  1.6]{Hansen17} it suffices to show that, for every cohomology class $\alpha
  \in H^i_{\mathrm{\acute{e}t}}(N^{\circ}, \mathbb{L}|_{N^{\circ}})$, there is a finite \'{e}tale
  covering $N'^{\circ} \xrightarrow{g} N^{\circ}$ such that \(g^*(\alpha)=0\).
  But this follows from \Cref{corokill}.
\end{proof}

\section{The site \(X_{\mathrm{prolog}}\)}
\label{section 3}

In this section we introduce the pro-log-\'{e}tale site \(X_{\mathrm{prolog}}\)
of the pair \((X,D)\) and show a comparison theorem between it and the previous site
\(X_{\log}\). It is parallel to~\cite[Section 3]{Sch13} except we will use a categorical
way to introduce this site.

In the following, we denote by $\mathcal{C}$ a category which has arbitrary
finite projective limits and a distinguished terminal object $X$.

Let $\mathcal{C}_f$ be a wide (i.e.~lluf) subcategory of $\mathcal{C}$ such that
the morphisms of $\mathcal{C}_f$ are stable under the base change via any
morphism in $\mathcal{C}$, i.e.~if $W\rightarrow V \in \mathrm{Hom}_{\mathcal{C}_f}$,
then $W \times_V Z \rightarrow Z$ is in $\mathrm{Hom}_{\mathcal{C}_f}$ for any $Z
\rightarrow V$. For the category $\mathcal{C}$, we have a functor
$|-|_{\mathcal{C}} \colon \mathcal{C} \rightarrow Top$ from $\mathcal{C}$ to the
category of topological spaces such that
\[
  |A\times_B C|_{\mathcal{C}}\rightarrow
  |A|_{\mathcal{C}}\times_{|B|_{\mathcal{C}}} |C|_{\mathcal{C}}
\]
is surjective with finite fibers for any maps $A\rightarrow B$ and $C\rightarrow
B$ in $\mathcal{C}$. Consider the pro-category pro-$\mathcal{C}$ of
$\mathcal{C}$. The functor $|-|_{\mathcal{C}}$ extends to a functor from
pro-$\mathcal{C}$ to $Top$ by $|\varprojlim N_i|=\varprojlim
|N_i|_{\mathcal{C}}$, and we denote it by $|-|$.

In the category of pro-$\mathcal{C}$, we can define several types of
morphisms.

\begin{definition}
\label{C-map}
  Let $W\rightarrow V$ be a morphism of pro-$\mathcal{C}$. We say
  $W\rightarrow V$ is a $\mathcal{C}$ map (resp.~$\mathcal{C}_f$ map) if
  $W\rightarrow V$ is induced by a morphism $W_0\rightarrow V_0$ in
  $\mathcal{C}$ (resp.~$\mathcal{C}_f$), i.e.~$W=V\times_{V_0} W_0$ via some map
  $V\rightarrow V_0$.

  We say that $W\rightarrow V$ is surjective if the corresponding map
  $|W|\rightarrow |V|$ is surjective. We say $W\rightarrow V$ is a
  pro-$\mathcal{C}$ map if $W=\varprojlim W_j$ can be written as a
  cofiltered inverse limit of $\mathcal{C}$ maps $W_i\rightarrow V$ over $V$ and
  $W_j\rightarrow W_i$ are surjective $C_f$ maps for large $j>i$. Note that $W_i$
  is an object of pro-$\mathcal{C}$. The presentation $W=\varprojlim
  W_i$ is called a pro-$\mathcal{C}$ presentation.

  We define a full subcategory $X_{\mathrm{pro}\mathcal{C}}$ of
  $\mathrm{pro}$-$\mathcal{C}$. The object of this category consists of objects
  in pro-$\mathcal{C}$ which are pro-$\mathcal{C}$ over $X$,
  i.e.~each object has a pro-$\mathcal{C}$ map to $X$. The morphisms
  are pro-$\mathcal{C}$ maps.
\end{definition}

The following lemma is almost identical to~\cite[Lemma 3.10]{Sch13}, except we
do not state the seventh sub-statement (which is the only non-formal one) here.

\begin{lemma}
\label{lemmascholze}
\leavevmode
\begin{enumerate}
\item Let $W\rightarrow V$ be a surjective morphism in $\mathcal{C}$. For any
  morphism $W'\rightarrow V$ in $\mathcal{C}$, the base change $W'\times_V W
  \rightarrow W'$ is surjective.
\item Let $W\rightarrow V$ be a $\mathcal{C}$ map (resp.~$\mathcal{C}_f$ map,
  resp.~pro-$\mathcal{C}$ map) in pro-$\mathcal{C}$. For any morphism
  $W'\rightarrow V$ in pro-$\mathcal{C}$, the base change $W'\times_V
  W\rightarrow W'$ is a $\mathcal{C}$ map (resp.~$\mathcal{C}_f$ map,
  resp.~pro-$\mathcal{C}$ map) and the map $|W'\times_V W| \rightarrow
  |W'|\times_{|V|} |W|$ is surjective, in particular, $W' \times_V W \rightarrow
  W'$ is surjective if $W'\rightarrow V$ is.
\item A composition of $E \rightarrow F \rightarrow G$ of two $\mathcal{C}$ maps
  (resp.~$\mathcal{C}_f$ maps) in pro-$\mathcal{C}$ is a $\mathcal{C}$
  map (resp.~$\mathcal{C}_f$ map).
\item A surjective $\mathcal{C}$ map (resp.~$\mathcal{C}_f$ map) $W\rightarrow
  V$ with $V\in X_{\mathrm{pro}\mathcal{C}}$ comes from a pull back via
  $V\rightarrow V_0$ from a surjective map $W_0\rightarrow V_0$ with $W_0,V_0
  \in \mathcal{C}$.
\item Let $E\rightarrow F\rightarrow G\rightarrow X$ be a sequence of morphisms where all the
  arrows are pro-$\mathcal{C}$ maps. Then $E,F \in X_{\mathrm{pro}\mathcal{C}}$
  and the composition $E\rightarrow G$ is a pro-$\mathcal{C}$ map.
\item If all maps in $\mathcal{C}$ have open images, then any pro-$\mathcal{C}$
  map $W \rightarrow V$ in pro-$\mathcal{C}$ has open image.
\end{enumerate}
\end{lemma}

\begin{proof}
\leavevmode
\begin{enumerate}
\item It follows from the surjectivity of the map \(|W' \times_V
  W|_{\mathcal{C}} \rightarrow |W'|_{\mathcal{C}}\times_{|V|_{\mathcal{C}}}
  |W|_{\mathcal{C}}\).

\item If $W\rightarrow V$ is a $\mathcal{C}$ map (resp.~$\mathcal{C}_f$ map)
  then by definition we reduce to the case $W,V\in \mathcal{C}$. Write
  $W'=\varprojlim W'_i$ with a compatible system of maps $W'_i\rightarrow V \in
  \mathcal{C}$. Then $W\times_V W'=\varprojlim W\times_V W'_i$ and $W\times_V W'
  \rightarrow W'$ is by definition again a $\mathcal{C}$ map
  (resp.~$\mathcal{C}_f$ map). As for the topological spaces, we have
  \[
    |W'\times_V W|=\varprojlim |W\times_V W'_i|\rightarrow \varprojlim
    |W|\times_{|V|} |W'_i|=|W|\times_{|V|} |W'|
  \]
  where the first equality follows from definition, and the last equality is due
  to that fiber products commute with inverse limits. The middle map is
  surjective because it is surjective with compact fibers at each finite stage,
  and inverse limits of nonempty compact spaces are nonempty. Actually, the
  fibers are nonempty compact spaces.

  In the general case, take a pro-\(\mathcal{C}\) presentation \(W'=\varprojlim
  W'_i \rightarrow V\). Then we have that $W'\times_V W=\varprojlim W'_i\times_V
  W \rightarrow W'$ is a pro-$\mathcal{C}$ map over $W'$ by what we have just
  proved. The map
  \[
    |W'\times_V W|=\varprojlim |W\times_V W'_i|\rightarrow \varprojlim
    |W|\times_{|V|} |W'_i|=|W|\times_{|V|} |W'|
  \]
  is surjective by the same reasoning as before.

\item Write $F=F_0\times_{G_0} G$ as a pullback of a $\mathcal{C}$ map
  (resp.~$\mathcal{C}_f$ map) $F_0\rightarrow G_0$. Moreover, write
  $G=\varprojlim G_i$ with a compatible system of maps $G_i\rightarrow G_0\in
  \mathcal{C}$. Then $F=\varprojlim ( G_j\times_{G_0} F_0)$. 

  Moreover write $E=F\times_{F'_0} E_0$ as a pullback of a $\mathcal{C}$ map
  (resp.~\(\mathcal{C}_f\) map) $ E_0\rightarrow F'_0$. Therefore the map
  $F\rightarrow F'_0$ factors through $ G_j\times_{G_0} F_0\rightarrow F'_0$ for
  large $j$. It follows that
  \[
    E=F\times_{ (G_j\times_{G_0} F_0) }\bigg((G_j\times_{G_0} F_0)\times_{F'_0}
    E_0\bigg)
  \]
  and $E\rightarrow G$ is a pullback of a $\mathcal{C}$ map
  (resp.~$\mathcal{C}_f$ map).

\item Let $V=\varprojlim V_i$ be a pro-$\mathcal{C}$ presentation over $X$. Note
  that $W\rightarrow V$ is induced by a pullback of a morphism $W_0\rightarrow
  V_0$ with $W_0,V_0\in \mathcal{C}$ via some map $V\rightarrow V_0$. The map
  $V\rightarrow V_0$ factors throguh a map $V_j\rightarrow V_0$ for large $j$.
  Therefore, we have $W=W_0\times_{V_0} V=(W_0\times_{V_0} V_j)\times_{V_j} V$.
  On the other hand, $|V|\rightarrow |V_j|$ is surjective for large $j$.
  Therefore $W_0\times_{V_0} V_j\rightarrow V_j$ is surjective.

\item One can write $E\rightarrow F$ as the composition $E\rightarrow
  E_0\rightarrow F$ of an inverse system $E=\varprojlim E_i\rightarrow E_0$ of
  surjective $\mathcal{C}_f$ maps $E_i\rightarrow E_j \rightarrow E_0$, and a
  $\mathcal{C}$ map $E_0\rightarrow F$. We check the statement separately in the
  case that $E\rightarrow F$ is a $\mathcal{C}$ map or an inverse system of
  surjective $\mathcal{C}_f$ maps. Assume that $E\rightarrow F$ is a
  $\mathcal{C}$ map which is induced by a map $E_0\rightarrow F_0 \in
  \mathcal{C}$, i.e. $E=F\times_{F_0} E_0$ via some map $F\rightarrow F_0$.
  Write $F=\varprojlim F_i \rightarrow G$ as a pro-$C$ presentation. Therefore,
  $F\rightarrow F_0$ factors through $F_i\rightarrow F_0$ for large $i$. It
  follows from (2) and (3) that $E=F\times_{F_0} E_0=\varprojlim (F_i\times_{
    F_0} E_0)\rightarrow G$ is a pro-$\mathcal{C}$ presentation over $G$, in
  other words, the composition $E\rightarrow G$ is a pro-$\mathcal{C}$ map.

  So it reduces us to consider all maps $E\rightarrow F\rightarrow G \rightarrow
  X$ are inverse systems of surjective $\mathcal{C}_f$ maps. Using (1) and (4),
  it is an easy exercise to show that all the compositions are still inverse
  systems of surjective $\mathcal{C}_f$ maps.
  
\item Let $U\rightarrow V$ be a pro-$\mathcal{C}$ map with a pro-$\mathcal{C}$
  presentation $U=\varprojlim U_i \rightarrow V$. Therefore we have
  $|U|\rightarrow |U_i| \rightarrow |V|$ with $|U|\rightarrow |U_i|$ surjective.
  It reduces us to show $|U_i|\rightarrow |V|$ has open image. Since
  $U_i\rightarrow V$ is a $C$ map, we have $U_i=V\times_{V_0} U_{i0}$ for some
  map $U_{i0}\rightarrow V_0$ in $\mathcal{C}$. By (2), we have a surjection
  $|U_i|=|V\times_{V_0} U_{i0}|\rightarrow |V|\times_{|V_0|} |U_{i0}|$.
  Therefore, the image of $|U_i|\rightarrow |V|$ is open.
\end{enumerate}
\end{proof}

We declare coverings in \(X_{\mathrm{pro}\mathcal{C}}\) as following: a covering
in $X_{\mathrm{pro}\mathcal{C}}$ is given by a family of pro-$\mathcal{C}$ maps
$\{f_i \colon V_i \rightarrow V\}$ such that $|V|=\bigcup_i |f_i| (|V_i|)$. From
\Cref{lemmascholze}, we know that $X_{\mathrm{pro}\mathcal{C}}$ is a site.

\begin{lemma}
\label{lemmacoverings}
Let $M\rightarrow N$ be a $\mathcal{C}$ map (see Definition~\ref{C-map}). 
If $M\rightarrow N$ is a
covering of $X_{\mathrm{pro}\mathcal{C}}$, then $M\rightarrow N$ is induced by a
covering $M_0\rightarrow N_0$ of $\mathcal{C}$.
\end{lemma}

\begin{proof}
  Write $N=\varprojlim N_i$ as a pro-$\mathcal{C}$ presentation of $N$. It
  follow that $|N|\rightarrow |N_i|$ is surjective for large $i$. Note that
  $M\rightarrow N$ is induced by a morphism $M'_0\rightarrow N'_0$ in
  $\mathcal{C}$ via some map $N\rightarrow N'_0$. The map $N\rightarrow N'_0$
  factors over $N_i\rightarrow N'_0$ for large $i$. Therefore, the map
  $M\rightarrow N$ is induced by the map $ N_i\times_{N'_0}M'_0\rightarrow N_i$.
  Hence $N_i \times_{N'_0} M'_0 \rightarrow N_i$ is a covering, and we may
  choose \(N_i\) (resp.~\(N_i \times_{N'_0} M'_0\)) to be the \(N_0\) (resp.~\(M_0\)) we
  are looking for.
\end{proof}

\begin{example}
  We can take $\mathcal{C}=X_{\mathrm{\acute{e}t}}$ and $\mathcal{C}_f$ to be the wide
  subcategory only allowing finite \'{e}tale maps to be morphisms. The functor
  $|-|_{\mathcal{C}}$ is the functor associating to an object its underlying
  topological space. Then \(X_{\mathrm{pro}\mathcal{C}}\) is just the
  pro-\'{e}tale site \(X_{\mathrm{pro}{\mathrm{\acute{e}t}}}\) introduced
  in~\cite{Sch13}.
\end{example}

\begin{example}
  Now we specialize our construction above to the Faltings site
  $\mathcal{C}=X_{\log}$, $|(V,N)|_{\mathcal{C}}=|N|$ for $(V,N)\in X_{\log}$
  and we take $\mathcal{C}_f$ to be the wide subcategory
  only allowing morphisms of the form $(V,N) \rightarrow (V,N')$, namely, every
  morphism of $\mathcal{C}_f$ is a morphism in $V_{\mathrm{f,log}} \cong
  V^{\circ}_{\mathrm{f}\mathrm{\acute{e}t}}$ for some $V \in X_{\mathrm{\acute{e}t}}$. Recall that
  a fiber product of morphisms $(V',N')\rightarrow (V,N)$ and $(V'',N'')\rightarrow (V,N)$ 
  in $X_{\log}$ is given by
  \begin{equation}
    \label{fiberproduct}
    \bigg( W=V'\times_V V'', (pr_1^*N' \times_{N \times_V W} pr_2^* N'')^{\nu} \bigg)
  \end{equation} 
  where $pr_1: W\rightarrow V'$ and $pr_2: W\rightarrow V''$ are the natural
  projections. It is easy to check $\mathcal{C},\mathcal{C}_f$ and
  $|-|_{\mathcal{C}}$ satisfy our assumptions of previous results, hence produce
  a site $X_{\mathrm{prolog}}$. In this site, we will call a pro-$\mathcal{C}$
  map (resp. $\mathcal{C}$ map, $\mathcal{C}_f$ map, pro-$\mathcal{C}$
  presentation) by pro-log-\'{e}tale map (resp.~log \'{e}tale map, finite log
  \'{e}tale map, pro-log-\'{e}tale presentation).
\end{example}

In concrete terms, the objects of $X_{\mathrm{prolog}}$ are of the form $(V,N)$ where 
$N=\varprojlim N_i$ is a tower of $N_i \xrightarrow{f_i} V$ such that $f_i$ is finite 
with $f_i^{\circ}$ \'{e}tale and \(N_i \to N_j\) is surjective for large \(i > j\). 
The space $|(V,N)|$ is given by
$\varprojlim |(V,N_i)|=\varprojlim |N_i|$. The category $X_{\mathrm{prolog}}$
has a natural fibered category structure over $X_{\mathrm{\acute{e}t}}$, namely we have a
natural functor $X_{\mathrm{prolog}} \rightarrow X_{\mathrm{\acute{e}t}}$ sending $(V,N)$
to $V$, and associating a morphism $V\xrightarrow{p} V'$ in $X_{\mathrm{\acute{e}t}}$ the
pullback map sending $N'=\varprojlim N'_i$ to $p^*(N')=\varprojlim p^*(N_i') = \varprojlim (N_i' \times_{V'} V)$.

If there is no confusion seemingly to arise, we will denote an object $(V,N) \in
X_{\mathrm{prolog}}$ by $N$.

\begin{lemma}
\label{lemmaxprolog}
\leavevmode
\begin{enumerate} 
\item The category $X_{\mathrm{prolog}}$ has arbitrary finite projective limits.
\item We have $\pi((V',N')\times_{ (V,N)}(V'',N')) = V'\times_V V''$ where $\pi$
  is the fibered structure functor $X_{\mathrm{prolog}} \xrightarrow{\pi}
  X_{\mathrm{\acute{e}t}}$.
\item The pro-log-\'{e}tale morphisms in pro-$X_{\log}$ have open images.
\end{enumerate}
\end{lemma}

\begin{proof}
\leavevmode
\begin{enumerate}
\item It suffices to check that finite products and equalizers exist. The first
  case follows from \Cref{lemmascholze} which is formal. The non-formal prat is
  to check for equalizers and we need to use the fact that locally $|N|$ has
  only a finite number of connected components. In fact, suppose that
  $f,g:N'\rightarrow N$ are two morphism of $X_{\mathrm{prolog}}$. By (the proof
  of)~\cite[Corollary 6.1.8]{catsheaf}, we can write $N'=\varprojlim N'_i$ and
  $N=\varprojlim N_i$ as pro-log-\'{e}tale presentations with the same index
  category and maps $f_i,g_i:N'_i\rightarrow N_i$ such that $f=\varprojlim f_i$
  and $g=\varprojlim g_i$. Let $E_i$ be the equalizer of $f_i$ and $g_i$ in
  $X_{\log}$. We get the following diagram (cf.~\Cref{lemmopen0}):
  \[
    \xymatrix{
      E_i\ar[r]\ar[d] & N'_i \ar[d] \ar@<-.5ex>[r]_{g_i}
      \ar@<.5ex>[r]^{f_i} & N_i \ar[d] \\
      E_j\ar[r] & N'_j \ar@<-.5ex>[r]_{g_j} \ar@<.5ex>[r]^{f_j} & N_i
    }
  \]
  where $N'_i\rightarrow N'_j$ and $N_i\rightarrow N_j$ are surjective for large
  $i$. We may assume that $V_i$ and $V'_i$ are affinoids. Denote the image of
  $E_i$ in $N_j'$ by $E^i_j$. Note that $E^i_j$ is open and closed by
  \Cref{lemmopen0} (2). Since $N'_j$ has finitely many connected components, the
  image $E^i_j$ stabilizes for $i$ larger than some $i_j$. Hence we see
  $E^\infty_j = E^{i_j}_j$ is the equalizer that we are looking for.
\item This follows from the description of fiber product in $X_{\log}$.
\item It follows from \Cref{lemmopen0} (2) and \Cref{lemmascholze} (6).
\end{enumerate}
\end{proof}

\begin{lemma}
\leavevmode
\begin{enumerate}
\item For an object $(V,N)\in X_{\mathrm{prolog}}$, if $V$ is affinoid, then $(V,N)$
  is a quasi-compact object of $X_{\mathrm{prolog}}$.

\item The family of all objects $(V,N)\in X_{\mathrm{prolog}}$ with $V$ affinoid is
  generating $X_{\mathrm{prolog}}$, and stable under fiber products.

\item The topos $\mathrm{Sh}(X_{\mathrm{prolog}})$ is algebraic and all objects
  $(V,N)$ of $X_{\mathrm{prolog}}$ with $V$ affinoid are quasi-compact and
  quasi-separated.

\item An object $(V,N)\in X_{\mathrm{prolog}}$ is quasi-compact if and only if
  $|(V,N)|$ is quasi-compact.

\item An object $(V,N)\in X_{\mathrm{prolog}}$ is quasi-separated if and only if
  $|(V,N)|$ is quasi-separated.
\end{enumerate}
\end{lemma}

\begin{proof}
\leavevmode
\begin{enumerate}
\item It follows from \Cref{lemmaxprolog} (3) that an object $(W,M)$ of
  $X_{\mathrm{prolog}}$ is quasi-compact if $|(W,M)|$ is quasi-compact. If $V$
  is affinoid, we can write $N=\varprojlim N_i$ with $N_i$ affinoid. Moreover,
  the space $|N_i|$ is a spectral space and the transition maps are spectral.
  Hence the inverse limit $\varprojlim |N_i|$ is a spectral space, and in
  particular quasi-compact. It follows that $(V,N)$ is a quasi-compact object of
  $X_{\mathrm{prolog}}$.

\item For an object $(V,N)\in X_{\mathrm{prolog}}$, we can use affinoid objects
  to cover $V$, i.e., $V=\cup V_i$. It is clear that $\{(V_i, N|_{V_i})\}$ is a
  covering of $(V,N)$ in $ X_{\mathrm{prolog}}$. The family is obviously stable
  under fiber products.

\item By (2) and~\cite[VI Proposition 2.1]{SGA42}, the topos
  $\mathrm{Sh}(X_{\mathrm{prolog}})$ is locally algebraic (see~\cite[VI Definition
  2.3]{SGA42}) and all objects $(V,N)$ of $X_{\mathrm{prolog}}$ with $V$
  affinoid are quasi-compact and quasi-separated. We check the criterion
  of~\cite[VI Proposition 2.2 (ii bis)]{SGA42} by considering the class of
  $(V,N)$ as in (1) that $V\rightarrow X$ factors over an affinoid open subset
  $V_0$ of X. It consists of coherent objects and is still generating
  $X_{\mathrm{prolog}}$. Note that $(V,N)\times_{(X,X)} (V,N)=(V,N)
  \times_{(V_0,V_0)} (V,N)$ is an object as in (1) which is quasi-separated.

\item Without loss of generality we may assume $|N| \to |V|$ is surjective. Therefore, the space $|V|$ is quasi-compact if $|(V,N)|$ is quasi-compact. Use finitely many objects $(V_i,N_i)$ of form in (1) to cover $(V,N)$. Note that $(V_i,N_i)$ are quasi-compact by (1). It follows that $(V,N)$ is quasi-compact.
  Conversely, if $(V,N)$ is compact, then we can find finitely many $(V_i,N_i)$
  with $V_i$ affinoid cover $V$. Note that $|(V_i,N_i)|$ is quasi-compact by the
  proof of (1). It follows that $|(V,N)|$ is quasi-compact.

\item Cover $(V,N)$ by $(V_i, N|_{V_i})$ as in the proof of (2). It follows from~\cite[VI Colloary 1.17]{SGA42} that the object $(V,N)$ is quasi-separated if and only if $(V_i,
  N|_{V_i}) \times_{(V,N)} (V_j, N|_{V_j})$ is quasi-compact if and only if
  $|(V_i, N|_{V_i})| \times_{|(V,N)|} |(V_j, N|_{V_j})|$ is quasi-compact if and
  only if $|(V,N)|$ is quasi-separated.
\end{enumerate}
\end{proof}

There is a natural projection $\nu:\mathrm{Sh}(X_{\mathrm{prolog}}) \rightarrow
\mathrm{Sh}(X_{\log})$ induced by the morphism of sites $X_{\mathrm{prolog}} \rightarrow X_{\log}$ sending $(V,N)$ to the constant tower $(V, \varprojlim N)$.

\begin{lemma}
\label{lemmacohomology}
\leavevmode
\begin{enumerate}
\item Let $\mathcal{F}$ be an abelian sheaf on $X_{\log}$. For any quasi-compact
  and quasi-separated $(V,N)=(V,\varprojlim N_j)\in X_{\mathrm{prolog}}$ and any
  $i \geq 0$, we have
  \[
    H^i((V,N),\nu^*\mathcal{F})=\varinjlim H^i((V,N_j), \mathcal{F}).
  \]
\item Let $\mathcal{F}$ be an abelian sheaf on $X_{\log}$. The adjunction
  morphism $\mathcal{F}\rightarrow R\nu_*\nu^*(\mathcal{F})$ is an isomorphism.
\end{enumerate}
\end{lemma}

\begin{proof}
\begin{enumerate}
\item We may assume that $\mathcal{F}$ is injective and that $X$ is
  quasi-compact and quasi-separated. Let us work with the subsite
  $X_{\mathrm{prolog}qc} \subset X_{\mathrm{prolog}}$ consisting of
  quasi-compact objects; note that
  $\mathrm{Sh}(X_{\mathrm{prolog}qc})=\mathrm{Sh}(X_{\mathrm{prolog}})$. Define
  a presheaf $G((V,N))=\varinjlim \mathcal{F}((V,N_i)) $ where $N=\varprojlim N_i$ with
  $N_i\in V_{\mathrm{f,log}}$. It is clear that $\nu^*\mathcal{F}$ is the sheaf associated
  to $G$. It suffices to show $G$ is a sheaf with $H^i((V,N),G)=0$ for all
  $(V,N)\in X_{\mathrm{prolog}qc}$ and $i>0$. By~\cite[V Proposition 4.3 (i) and
  (iii)]{SGA42}, we just need to prove that for any $(V,N)\in X_{\mathrm{prolog}qc}$
  with a pro-log-\'{e}tale covering $(V_k,N_k) \rightarrow (V,N)$ in $
  X_{\mathrm{prolog}qc}$, the correspondding Cech complex
  \[
    0\rightarrow G((V,N)) \rightarrow \prod \limits_k G(( V_k,N_k)) \rightarrow
    \prod\limits_{k,k'} G((V_k,N_k)\times_{(V_k,N_k)} (V_{k'},N_{k'}))
    \rightarrow \ldots
  \]
  is exact. This shows that $G$ is a sheaf and then all higher cohomology groups
  vanish.

  Since $(V,N)$ is quasi-compact, we can pass to a finite subcover and combine
  them into a single morphism $(V',N') \xrightarrow{(p,q)} (V,N)$. Write it in a
  pro-log-\'{e}tale presentation $N'=\varprojlim N'_i\rightarrow N$. In the following,
  we write the Cech complex of $G$ with respect to the covering $(p,q)$ as
  $\mathrm{Cech}(N'\rightarrow N)$. Therefore, we
  have 
\[
\mathrm{Cech}(N'\rightarrow N)=\varinjlim \mathrm{Cech}(N'_i \rightarrow N) 
\]  
where $N'_i\rightarrow N$ is a covering for large $i$. Therefore, it
  suffices to show the exactness of $\mathrm{Cech}(N'_i\rightarrow N)$. By
  \Cref{lemmacoverings}, the cover $N'_i\rightarrow N$ is induced by a cover
  $N'_0\rightarrow N_0$ in $X_{\log}$, i.e. $N'_i= N'_0\times_{N_0} N$.
  Therefore, $\mathrm{Cech}(N'_i\rightarrow N)$ is the direct limit of the Cech
  complexes for some covers in $X_{\log}$. But this is acyclic by the
  injectivity of $G$ on $X_{\log}$.

\item Note that $R^i\nu_*\nu^*\mathcal{F}$ is the sheaf on $X_{\log}$ associated
  to the presheaf $(V,N) \mapsto H^i((V,N), \nu^*\mathcal{F})$ where $(V,N)$ is
  considered as an element of $X_{\mathrm{prolog}}$. Hence, (1)
  says that we get an isomorphism for $i=0$. Moreover, for $i$ positive, (1) says that $ H^i((V,N), \nu^*\mathcal{F})=H^i((V,N), \mathcal{F})$
  if $(V,N) \in X_{\log}$ is quasi-compact and quasi-separated. By the local acyclicity of
  higher cohomology group, a section of $H^i((V,N), \mathcal{F})$ vanishes
  locally in the topology $X_{\log}$, so the associated sheaf is trivial. It
  follows that $R^i\nu_*\nu^*\mathcal{F}=0$ for $i>0$.
\end{enumerate}
\end{proof}

\section{The structure sheaves}
\label{section 4}

In this section, parallel to~\cite[Section 4]{Sch13}, we introduce the structure
sheaves \(\mathcal{O}^+\), \(\mathcal{O}\), \(\hat{\mathcal{O}}^+\) and
\(\hat{\mathcal{O}}\) on our site \(X_{\mathrm{prolog}}\). In the following we will not distinguish rigid spaces and their associated adic spaces.

\begin{definition}
  With the notations as in \Cref{defux}, let $X$ be a rigid space over $\mathrm{Sp}(k)$
  with an SSNC divisor $D$. Consider the following sheaves on $X_{\log}$ and $X_{\mathrm{prolog}}$.
\begin{enumerate}
\item The integral structure sheaf $\mathcal{O}_{X_{\log}}^+$ on $X_{\log}$ is
  given by $u_{X,*}(\mathcal{O}_{U_{\mathrm{\acute{e}t}}}^+)$. By~\cite[Theorem 2.6]{Hansen17} we have
  $\mathcal{O}_{X_{\log}}^+\left((V,N)\right)=\mathcal{O}_N^+(N)$ for an object
  $(V,N)\in X_{\log}$. The structure sheaf
  $\mathcal{O}_{X_{\log}}$ on $X_{\log}$ is given by
  $\mathcal{O}_{X_{\log}}^+[\frac{1}{p}]=
  u_{X,*}(\mathcal{O}_{U_{\mathrm{\acute{e}t}}}^+)[\frac{1}{p}]$, namely,
  $\mathcal{O}_{X_{\log}}(V,N)=\mathcal{O}_N(N)$ for quasi-compact and
  quasi-separated $(V,N)$.

\item The (uncompleted) structure sheaf is defined to be
  $\nu^*\mathcal{O}_{X_{\log}}$ on $X_{\mathrm{prolog}}$ with
  subring of integral elements $\nu^*\mathcal{O}_{X_{\log}}^+$. 
  If no confusion seems to arise, we will still denote them by $\mathcal{O}_{X_{\log}}$
  and $\mathcal{O}_{X_{\log}}^+$ respectively.

\item We define the completed integral structure sheaf (on $X_{\mathrm{prolog}}$) to be $\hat{\mathcal{O}}_{X_{\log}}^+=\varprojlim
  \mathcal{O}_{X_{\log}}^+/p^n$, and the completed structure is defined as
  $\hat{\mathcal{O}}_{X_{\log}}=\hat{\mathcal{O}}_{X_{\log}}^+[\frac{1}{p}]$.
\end{enumerate}
\end{definition}

For simplicity, for the rest of this section we assume $X$ is a rigid space over a perfectoid field $K$.

\begin{definition}
\label{defaffpert}
  Let $(V,N) \in X_{\mathrm{prolog}}$ with $V \xrightarrow{f} X$. We say that 
  $(V,N)$ is affinoid perfectoid if
\begin{enumerate}
\item $V$ is affinoid with $V=\mathrm{Sp}(R')$ and $f^{-1}(D_i)$ is cut out 
by one equation $f_i$;
\item $N$ has a presentation $N=\varprojlim N_j$ for a cofiltered system
  $\{N_j=\mathrm{Sp}(R_j)\}$ of objects in $V_{\mathrm{f,log}}$ such that 
  \begin{itemize}
  \item $N_j$ are smooth;
  \item $\{N_j\}$ contains a cofiltered subsystem consisting of all branched coverings
  $\mathrm{Sp}\left(R'[\sqrt[k]{f_i}]\right)$ for all $k\in
  \mathbb{N}$ and;
  \item denote by $R^+$ the $p$-adic completion of $\varinjlim R_i^{\circ} $,
  and $R=R^+[\frac{1}{p}]$, the pair $(R,R^+)$ is a perfectoid affinoid
  $(K,K^{\circ})$-algebra.
  \end{itemize}
\end{enumerate}
\end{definition}

\begin{remark}
In the above definition (2), one can actually drop the first condition. Indeed, any cofiltered system
satisfying second condition automatically has a cofinal subsystem with $N_j$ being smooth by 
\Cref{Abhyankar's Lemma}.
\end{remark}

\begin{warning}
In the following, the index of an inverse system in $X_{\mathrm{prolog}}$
is denoted by $i$. For instance, we will write $N = \varprojlim N_i$.
But the readers should be cautious that the $i$ here is not the same as the index of $f_i$'s.
Since these two indexing systems have very different meanings, 
we hope no confusions should arise from this.
\end{warning}

We say that $(V,N)$ is perfectoid if it has an open cover by affinoid
perfectoid. To an affinoid perfectoid $(V,N)$ as above, we can associate
$\hat{N}=\mathrm{Spa}(R,R^+)$ which is an affinoid perfectoid space over $\mathrm{Spa}(K,\mathcal{O}_K)$.
One immediately checks that this is well-defined, i.e.~independent of the presentation of
$N=\varprojlim N_i$. Moreover, we have $\hat{N} \sim \varprojlim N_i $ in the
sense of~\cite[Definition 7.14]{Scholze1}, in particular $|\hat{N}|=|N|$.

\begin{example}
\label{model example}  
Take 
\[
X=V=\mathrm{Sp}\left(K \langle Z_1^{\pm 1},\ldots, Z_{n-r}^{\pm 1},
  Z_{n-r+1},\ldots, Z_{n} \rangle \right)=\mathbb{T}^{n-r} \times \mathbb{D}^{r},
\]
denote it by $\mathbb{T}^{n-r,r}$, with the divisor $D$ given by $Z_{n-r+1}
\cdots Z_n=0$. Then $(\mathbb{T}^{n-r,r},N) \in X_{\mathrm{prolog}}$ with
$N=\tilde{\mathrm{T}}^{n-r,r}$ being the inverse limit of the
\[
  \mathrm{Sp}\left(K \langle Z_1^{\pm 1/p^k},\ldots, Z_{n-r}^{\pm 1/p^k},
    Z_{n-r+1}^{1/l},\ldots, Z_{n}^{1/l} \rangle \right)
\]
is an affinoid perfectoid. Using the notations from discussion before this example, we have
\[
R = K \langle Z_1^{\pm 1/p^{\infty}}, \ldots, Z_{n-r}^{\pm 1/p^{\infty}}, Z_{n-r+1}^{1/\infty}, 
\ldots, Z_n^{1/\infty} \rangle
\]
and
\[
R^+ = \mathcal{O}_K \langle Z_1^{\pm 1/p^{\infty}}, \ldots, Z_{n-r}^{\pm 1/p^{\infty}}, Z_{n-r+1}^{1/\infty}, 
\ldots, Z_n^{1/\infty} \rangle.
\]
\end{example}

The following lemma is an analogue of~\cite[Lemma 4.5]{Sch13}. The proof is exactly the same.

\begin{lemma}
\label{lemmafet}
With the notations as in \Cref{defaffpert}, let $(V,N) \in X_{\mathrm{prolog}}$ be an
affinoid perfectoid with $N=\varprojlim N_i$ and $N_i=\mathrm{Spa}(R_i, R_i^{\circ})$ so that
$\hat{N}=\mathrm{Spa}(R,R^+)$.

Assume that $M_i=\mathrm{Spa}(S_i, S_i^{\circ}) \rightarrow N_i$ is an \'{e}tale map which can be
written as a composition of rational subsets and finite \'{e}tale maps. For $ j \geq
i$, write $M_j = M_i \times_{N_i} N_j = \mathrm{Spa}(S_j,S_j^{\circ})$ and $M = M_i\times _{N_i}
N=\varinjlim M_j \in \mathrm{pro}$-({Rigid}/$M_i$). Let $A_j$ be the $p$-adic
completion of the p-torsion free quotient of $S_j^{\circ} \otimes_{R_j^{\circ}}R^+$. Then
\begin{enumerate}
\item The completion $(S,S^+)$ of the direct limit of the $(S_j,S_j^{\circ})$ is a
  perfectoid affinoid $(K,K^{\circ})$-algebra. Moreover, $\hat{M} = M_j \times _{M_j}
  \hat{N}$ in the category of adic spaces over $\mathrm{Spa}(K,K^{\circ})$, and $S=A_j[\frac{1}{p}]$ for
  any $j\geq i$, where $\hat{M}$ is similarly defined as $\hat{N}$, i.e.
  $\hat{M}=\mathrm{Spa} \left( (\hat{\varinjlim S_j^{\circ} })[\frac{1}{p} ], \hat{\varinjlim S_j^{\circ} }
  \right)$.
\item For any $j\geq i$, the cokernel of the map $A_j \rightarrow S^+$ is
  annihilated by some power $p^N$ of p.
\item Let $\epsilon > 0$, $\epsilon \in \log\Gamma$. Then there exists some $j$ such that the cokernel of the map $A_j \to S^+$ is annihilated by $p^{\epsilon}$.
\end{enumerate}
\end{lemma}

\begin{proof}
The proof is the same as~\cite[Lemma 4.5]{Sch13}. Roughly speaking, for $M_i \subseteq N_i$ being a rational subset, it follows from the property that a rational subset of an affinoid perfectoid space is affinoid perfectoid, see \cite[Theorem 6.3 (ii)]{Scholze1}. 
For $M_i \to N_i$ being a finite \'{e}tale morphism, it follows from the almost purity theorem \cite[Theorem 7.9(iii)]{Scholze1}.
\end{proof}

\begin{lemma} 
\label{lemmaloget}
Let $(V,N')\rightarrow (V,N)$ be a finite log \'{e}tale morphism in $X_{\mathrm{prolog}}$. If $(V,N)$ is affinoid perfectoid, then the morphism $(V,N')\rightarrow (V,N)$ is induced by a finite \'{e}tale morphism between two objects of $V_{\mathrm{f,log}}$, i.e. $N'=N\times_{N_0} N'_0 $ via some finite \'{e}tale morphism $N'_0\rightarrow N_0$, and $(V,N')$ is affinoid perfectoid.
\end{lemma}

\begin{proof}
Use the notations as in \Cref{defaffpert}. Suppose that 
$(V,N') \rightarrow (V,N)$ is induced by $(V,N'_0)\rightarrow (V,N_0)$ in 
$V_{\mathrm{f,log}}$, i.e. $(V,N')=(V,N'_0) \times_{(V,N_0)} (V,N)$ via some map 
$N\rightarrow N_0$ of pro-$V_{\mathrm{f,log}}$ where $N_0$ is smooth. By Lemma \ref{Abhyankar's Lemma}, we know there is $N_0[\sqrt[k]{f_i}]=N_1 \rightarrow N_0$ for large $k$ such that $N_1' \coloneqq N_1 \times_{N_0} N_0' \to N_1$ is finite \'{e}tale. 
Now by our assumption of $(V,N)$ being affinoid perfectoid, we may find $N_2$ inside the tower of $N$ dominating $N_1$. 
Therefore $N'$ is induced by the morphism $N_2' \coloneqq N_2 \times_{N_0} N_0' \to N_2$ which is finite \'{e}tale. One checks $(V,N')$ is an affinoid perfectoid: it consists of cofinal system of smooth $N_j'$'s since $N'$ is induced by a finite \'{e}tale morphism; the completed algebra being perfectoid follows from almost purity (see~\cite[Theorem 7.9]{Scholze1}); and since our system has a subsystem dominating $\mathrm{Sp}(R'[\sqrt[k]{f_i}])$, throwing them in our system gives rise to a presentation of $N'$.
\end{proof}

\begin{theorem}
\label{perfectoid basis}
The set of $(V,N)\in X_{\mathrm{prolog}}$ which are affinoid perfectoid form a basis for the topology.
\end{theorem}

\begin{proof}
Use the notations as in Example \ref{model example}. If $(X,D)=(\mathbb{T}^{n-r,r},D)$, then we have made an explicit cover of $X$ by an affinoid perfectoid $\tilde{\mathrm{T}}^{n-r,r} \in X_{\mathrm{prolog}}$.
Let $(V,N)$ be an object of $X_{\mathrm{prolog}}$ with $V \to X$ \'{e}tale, $N=\varprojlim N_i \xrightarrow{h} V$ where $N_i\xrightarrow{h_i}V \in V_{\mathrm{f,log}}$. 
By~\cite[Theorem 1.18]{Kiehl67} and~\cite[Corollary 1.6.10]{Huber96}, we may assume that $V$ admits an \'{e}tale morphism $V \xrightarrow{f} \mathbb{T}^{n-r,r}$ with divisor given by $f^{-1}(D)$. 
We may further assume that $f$ is the composite of a rational open embedding and a finite \'{e}tale morphism. Therefore, $(V,f^{*}(\tilde{\mathrm{T}}^{n-r,r}))=(V,V\times_{\mathbb{T}^{n-r,r}}\tilde{\mathrm{T}}^{n-r,r})\in X_{\mathrm{prolog}}$ is affinoid perfectoid by Lemma \ref{lemmafet}. By Lemma \ref{lemmaloget}, we know $(V, h_i^{*}(f^{*}(\tilde{\mathrm{T}}^{n-r,r}))$ is also affinoid perfectoid. Note that 

\begin{enumerate}
\item $N\times_V f^{*}(\tilde{\mathrm{T}}^{n-r,r})=h^{*}(f^{*}(\tilde{\mathrm{T}}^{n-r,r}))=\varprojlim h_i^{*}(f^{*}(\tilde{\mathrm{T}}^{n-r,r})) $ and;
\item the completion of a direct limit of perfectoid affinoid $(K,K^{\circ})$-algebra is again perfectoid affinoid.
\end{enumerate}   Therefore $(V, N\times_V f^{*}(\tilde{\mathrm{T}}^{n-r,r}))$ is affinoid perfectoid which covers $(V,N)$.
\end{proof}

\begin{lemma} 
\label{lemmascholze410}
Assume that $(V,N)\in X_{\mathrm{prolog}}$ is affinoid perfectoid with $\hat{N}=\mathrm{Spa}(R,R^+)$. 
\begin{enumerate}
\item For any nonzero element $b\in K^{\circ}$, we have $\mathcal{O}_{X_{\log}}^+((V,N))/b=R^+/b$ and it is almost equal to $(\mathcal{O}^+_{X_{\log}}/b)((V,N))$.
\item The image of $(\mathcal{O}^+_{X_{\log}}/b_1)((V,N))$ in $(\mathcal{O}^+_{X_{\log}}/b_2)((V,N))$ is equal to $R^+/b_2$ for any nonzero nilpotent elements $b_1, b_2 \in K^{\circ}$ with $|b_1|<|b_2|$
\item We have $\hat{\mathcal{O}}_{X_{\log}}^{+}((V,N))=R^+$ and $\hat{\mathcal{O}}_{X_{\log}}((V,N))=R$.
\item The ring $\hat{\mathcal{O}}_{X_{\log}}^{+}((V,N))$ is the $p$-adic completion of $\mathcal{O}_{X_{\log}}^{+}((V,N))$.
\item The cohomology groups $H^i((V,N),\hat{\mathcal{O}}^+_X) $ are almost zero for $i>0$.

\end{enumerate}
\end{lemma}

\begin{proof}
The proof is almost identical to the proof of~\cite[Lemma 4.10]{Sch13}. We sketch the proof for the sake of completeness. 
As in the proof of~\cite[Lemma 4.10]{Sch13}, it suffices to show that $N\rightarrow \mathcal{F}(N)=(\mathcal{O}^+_{\hat{N}}(\hat{N})/b)^a=(\mathcal{O}^+_{N}(N)/b)^a$ is a sheaf of almost $K^{\circ}$-algebra, with $H^i(N,\mathcal{F})=0$ for $i>0$. 

Let $N$ be a quasi-compact object being covered by $N_k\rightarrow N$. 
By quasi-compactness of $N$, we can assume that the covering consists of only one pro-log-\'{e}tale morphism $N'\rightarrow N$. 
Write $N'=\varprojlim N'_i \rightarrow N'_0\rightarrow N$, where $N'_0\rightarrow N$ is log-\'{e}tale morphism and $N'_i\rightarrow N'_j$ is surjective finite log-\'{e}tale for $i>j \geq 0$. 
Note that the morphism $q$ of a morphism $(W,M)\xrightarrow{(p,q)} (V,M')$ of $X_{\log}$ can be written as a composition of an \'{e}tale morphism and a morphism of $W_{\mathrm{f,log}}$, e.g. $M\rightarrow p^*M'\rightarrow M'$. 
Therefore, we can assume that $N'_0\rightarrow N$ is induced by an \'{e}tale morphism $V'_0 \rightarrow V$ of $X_{\mathrm{\acute{e}t}}$.
Furthermore, by Lemma \ref{lemmaloget}, the morphisms $N'_i\rightarrow N'_j $ are induced by finite \'{e}tale morphisms of $(\pi (N'_j))_{\mathrm{f,log}}$. 

On the other hand, we have to show that the complex
\[
\mathcal{C} (N',N):  0 \rightarrow \mathcal{F}(N)\rightarrow \mathcal{F}(N')\rightarrow \mathcal{F}(N'\times_N N')\rightarrow \ldots 
\] 
is exact. Note that $\mathcal{F}(N')=\varinjlim \mathcal{F}(N'_j).$ So we have 
\[
\mathcal{C} (N',N)=\varinjlim \mathcal{C} (N'_i,N).
\]
and one reduces to the case that $N'\rightarrow N$ is a composite of rational embeddings and finite \'{e}tale maps. In this case, both $N$ and $N'$ are affinoid perfectoid, giving rise to perfectoid spaces $\hat{N'}$ and $\hat{N}$, and an \'{e}tale cover $\hat{N'}\rightarrow \hat{N}$. Then Lemma \ref{lemmafet} implies that 
\[
\mathcal{C} (N',N):  0 \rightarrow (\mathcal{O}^+_{\hat{N}_{\mathrm{\acute{e}t}}}(\hat{N})/b)^a\rightarrow (\mathcal{O}^+_{\hat{N}_{\mathrm{\acute{e}t}}}(\hat{N'})/b)^a\rightarrow (\mathcal{O}^+_{\hat{N}_{\mathrm{\acute{e}t}}}(\hat{N'}\times_{\hat{N}}\hat{N'} )/b)^a \rightarrow \ldots 
\] 
is exact. Note that $\mathcal{F}(N')=\varinjlim \mathcal{F}(N'_j)$. 
Therefore the statement follows from the vanishing of $H^i(W_{\mathrm{\acute{e}t}},\mathcal{O}_{W_{\mathrm{\acute{e}t}}}^{+a})=0$ for $i>0$ and any affinoid perfectoid space $W$, see~\cite[Proposition 7.13]{Scholze1}.
\end{proof}

\begin{lemma}
\label{analogue Lemma 4.12}
Assume that \((V,N)\) is an affinoid perfectoid, with
\(\hat{N}=\mathrm{Spa}(R,R^+)\). Let \(\mathbb{L}\) be an \(\mathbb{F}_p\)-local
system on \(U = X \setminus D\). Then for all \(i > 0\), the cohomology group
\[
  H^i \bigg((V,N), \nu^* ( u_X (\mathbb{L}) \otimes \mathcal{O}_{X_{\log}}^+/p)\bigg)^a = 0,
\]
and it is almost finitely generated projective \(R^{+a}/p\)-module \(M(N)\) for
\(i = 0\). If \((V',N')\) is affinoid perfectoid, corresponding to \(\hat{N}' =
\mathrm{Spa}(R',R^{'+})\), and \((V',N') \to (V,N)\) some map in
\(X_{\mathrm{prolog}}\), then \(M(N') = M(N) \otimes_{R^{+a}/p} R'^{+a}/p\).
\end{lemma}

\begin{proof}
  We just need to notice that $\nu^* \bigg(u_{X,*} (\mathbb{L})\bigg)$ will be
  extended to an \(\mathbb{F}_p\)-local system on \(N_k\) for some \(k\) in the index category of $N$ (by \Cref{Abhyankar's Lemma}). Therefore it follows from~\cite[Lemma 4.12]{Sch13}.
\end{proof}

\section{Primitive Comparison}
\label{section 5}

Following~\cite[Section 5]{Sch13}, in this section we show the primitive
comparison in our setting.

\begin{theorem}
\label{thmpc}
  Let $K$ be an algebraically closed complete extension of $\mathbb{Q}_p$, and
  let $X$ be a proper smooth rigid analytic space over $\mathrm{Sp}(K)$ with an
  SSNC divisor $D$. Let $\mathbb{L}$ be an $\mathbb{F}_p$-local system on
  $(X-D)_{\mathrm{\acute{e}t}}$. Then there is an isomorphism of almost finitely
  generated $K^{\circ a}$-modules
  \[
    H^{i}(X_{\log},u_{X,*}(\mathbb{L})) \otimes_{\mathbb{F}_p}K^{oa}/p \cong
    H^i(X_{\log}, u_{X *}(\mathbb{L}) \otimes \mathcal{O}^+_{X_{\log}}/p)
  \]
  for all $i \geq 0$, where $u_X$ is defined in \Cref{defux}.
\end{theorem}

One can see the finiteness from the proof. We remark a more direct proof
which is based on the primitive comparison of Scholze and functorial embedded
resolution for rigid spaces over characteristic $0$ fields due to Temkin.

\begin{remark}
\label{finiteness remark}
\leavevmode
\begin{enumerate}
\item Assume $\mathbb{L}$ comes from an $\mathbb{F}_p$-local system on $X$. Then by~\cite[Theorem 5.1]{Sch13}, \Cref{Gysin} and \Cref{comparison logtopos} we have that $H^{i}(X_{\log},u_{X,*}(\mathbb{L}))$ is a finite dimensional $\mathbb{F}_p$ vector space for all $i \geq 0$, which vanishes for $i > 2 \dim X$.
\item In general, by~\cite[Theorem 1.6]{Hansen17} and~\cite[Theorem 1.1.13]{Tem17}, we can find a $U'$ finite \'{e}tale over $U$ such that
\begin{itemize}
\item $\mathbb{L}|_{U'} \cong \mathbb{F}_p^{\oplus r}$ and;
\item $U'$ admits a smooth compactification with complement divisor SSNC.
\end{itemize}
Hence we have the finiteness for $H^{i}(X_{\log},u_{X,*}(\mathbb{L}))$ as in (1).
\end{enumerate}
\end{remark}

\begin{lemma}
\label{analogue Lemma 5.2}
Let \(k\) be a complete nonarchimedean field. Let \(V\) be an affinoid smooth
adic space over \(\mathrm{Spa}(k,\mathcal{O}_k)\). Let \(D = \bigcup_{i=1}^l
D_i\) be an SSNC divisor in \(V\) and let \(x \in D_{i_1 i_2 \cdots i_r}
\backslash \bigcup_{j \in \{1,2,\ldots,l\}\backslash\{i_1, i_2, \ldots, i_r\}}
D_{i_1 i_2 \cdots i_r j}\) with closure \(M= \overline{\{x\}}\). Then there
exists a rational subset \(U \subset V\) containing \(M\), with \(U \cong S
\times \mathbb{D}^m(\underline{s})\), together with an \'{e}tale map \(S
\xrightarrow{\phi} \mathbb{T}^{n-r}\) satisfying the following two conditions:
\begin{enumerate}
\item \(\phi\) factors as a composite of rational embeddings and finite
  \'{e}tale maps and;
\item \(D_{i_j} \cap U\) is given by the vanishing locus of \(s_j\) and if \(i
  \notin \{i_1,i_2,\ldots,i_r\}\) then \(D_i \cap U = \varnothing\).
\end{enumerate}
\end{lemma}

\begin{proof}
  By~\cite[Theorem 2.11]{Mitsui} we may first find a rational subset \(U_0
  \subset V\) containing \(M\) such that \(U_0 \cong S_0 \times
  \mathbb{D}^r(\underline{s})\), where \(S_0\) is a smooth affinoid, satisfying
  our condition (2). Note that one can find such a rational containing \(M\)
  because of condition (1) of~\cite[Theorem 2.11]{Mitsui}.

  Now we may apply~\cite[Lemma 5.2]{Sch13} to our \((x, S_0 \times \{0\})\) to
  find a rational subset \(S \subset S_0\) together with \(\phi\) satisfying our
  condition (1).
\end{proof}

\begin{lemma}
\label{analogue Lemma 5.3}
Let \(k\) be a complete nonarchimedean field. Let \(X\) be a proper smooth adic
space over \(\mathrm{Spa}(k,\mathcal{O}_k)\). Let \(\bigcup_{i \in I} D_i = D
\subset X\) be an SSNC divisor where \(I = \{1,2,\ldots,r\}\). For any integer
\(N \geq 1\) and \(N\) distinct elements \(\gamma_N < \gamma_{N-1} < \cdots <
\gamma_1 = 1\) in the norm group \(\Gamma\) of \(k\), one may find \(N\) finite
covers \(\mathcal{U}^{(i)} = \bigcup_{J \subset I} \{U^{J,(i)}_l\}\) of \(X\) by
affinoid open subsets. Here \(U^{J,(i)}_l \cong S^{J,(i)}_l \times
\mathbb{D}^{\left\vert J \right\vert}(\underline{s}/p^{\gamma_i})\) where
\(S^{J,(i)}_l\) (viewed as \(S^{J,(i)}_l \times \{0\}\)) are affinoid open
subsets of \(D_J\), such that the following conditions hold:
\begin{enumerate}
\item \(D \cap U^{J,(i)}_l\) is given by vanishing locus of coordinates on the
  disc;
\item For all i, J and l, the closure \(\overline{S^{J,(i+1)}_l}\) of
  \(S^{J,(i+1)}_l\) in \(D_J\) is contained in \(S^{J,(i)}_l\). Hence the
  closure of \(U^{J,(i+1)}_l\) in \(X\) is contained in \(U^{J,(i)}_l\);
\item For all l and J, \(S^{J,(N)}_l \subset \ldots \subset S^{J,(1)}_l\) is a
  chain of rational subsets. Hence the same holds for \(U^{J,(i)}_l\)'s;
\item For J, J', l and l', the intersection \(U^{J,(1)}_l \cap U^{J',(1)}_{l'}
  \subset U^{J,(1)}_l\) is a rational subset and;
\item For all l and J, there is an \'{e}tale map \(S^{J,(1)}_l \to
  \mathbb{T}^{n-\left\vert J \right\vert}\) that factors as a composite of
  rational subsets and finite \'{e}tale maps.
\end{enumerate}
\end{lemma}

\begin{proof}
  The proof is almost identical to that of~\cite[Lemma 5.3]{Sch13} except we
  use~\Cref{analogue Lemma 5.2} to replace~\cite[Lemma 5.2]{Sch13} in the
  argument.
\end{proof}

\begin{lemma}
\label{analogue Lemma 5.5}
Let \(K\) be a complete non-archimedean field extension of \(\mathbb{Q}_p\) that
contains all roots of unity; choose a compatible system \(\zeta_l \in K\) of
\(l\)-th roots of unity. Let
\[
  R_0 = \mathcal{O}_K \langle T_1^{\pm 1},\ldots,T_{n-r}^{\pm 1}, T_{n-r+1},
  \ldots, T_n \rangle,
\]
\[
  R' = \mathcal{O}_K \langle T_1^{\pm 1/{p^\infty}},\ldots,T_{n-r}^{\pm
    1/{p^\infty}}, T_{n-r+1}^{1/\infty}, \ldots, T_n^{1/\infty} \rangle
\]
where \(T^{1/\infty}\) means adjoining all power roots of \(T\), and
\[
  R = \mathcal{O}_K \langle T_1^{\pm 1/{p^\infty}},\ldots,T_{n-r}^{\pm
    1/{p^\infty}}, T_{n-r+1}^{\pm1/{p^\infty}}, \ldots, T_n^{\pm 1/{p^\infty}}
  \rangle.
\]
Let \(S_0\) be an \(R_0\)-algebra which is \(p\)-adically complete flat over
\(\mathbb{Z}_p\) with the \(p\)-adic topology. Let \(\Delta \coloneqq
\mathbb{Z}_p^{n-r} \times \widehat{\mathbb{Z}}^{r}\) such that the \(k\)-th
basis vector acts on \(R'\) via
\[
  \prod T_j^{i_j} \mapsto \zeta^{i_k} \prod T_j^{i_j},
\]
where \(\zeta^{i_k} = \zeta_l^{i_k l}\) whenever \(i_k l \in \mathbb{Z}\). Let
\(\Delta \twoheadrightarrow \Delta_{\infty} \coloneqq \mathbb{Z}_p^n\) be the
obvious projection. Then
\begin{enumerate}
\item \(H^q_{cont}(\Delta_{\infty}, S_0/p^m \otimes_{R_0/p^m} R/p^m) \to
  H^q_{cont}(\Delta, S_0/p^m \otimes_{R_0/p^m} R'/p^m)\) is an almost
  isomorphism,
\item \(H^q_{cont}(\Delta_{\infty}, R/p^m)\) is an almost finitely presented
  \(R_0\)-module for all \(m\),
\item the map
\[
  \bigwedge^q R_0^n = H^q_{cont}(\Delta_{\infty},R_0) \to H^q_{cont}(\Delta, R')
  (~=^a H^q_{cont}(\Delta_{\infty}, R) \text{ by (1) above })
\]
is injective with cokernel killed by \(\zeta_p-1\),
\item \(H^q_{cont}(\Delta_{\infty}, S_0/p^m \otimes_{R_0/p^m} R/p^m) = S_0/p^m
  \otimes_{R_0/p^m} H^q_{cont}(\Delta_{\infty}, R/p^m)\) for all \(m\) and,
\item \(H^q_{cont}(\Delta_{\infty}, S_0 \hat{\otimes}_{R_0} R) = S_0
  \hat{\otimes}_{R_0} H^q_{cont}(\Delta_{\infty}, R)\)
\end{enumerate}
\end{lemma}

\begin{proof}
  (1) follows from~\cite[Lemma 3.10]{Olsson}, notice that the action of
  \(\Delta\) is continuous with respect to the \(p\)-adic topology on these
  \(\Delta\)-modules. (2) to (5) follows from (the proof of)~\cite[Lemma
  5.5]{Sch13}.
\end{proof}

\begin{lemma}
\label{analogue Lemma 5.6}
Let \(K\) be as in the previous Lemma. Let \((V,N)\) be an object in
\(X_{\log}\) with an \'{e}tale map \(V \to \mathbb{T}^{n-r,r}\) as one of
the \(V^{J,(1)}_k\)'s in \Cref{analogue Lemma 5.3}. Let \(\mathbb{L}\) be an
\(\mathbb{F}_p\)-local system on \(U_{\mathrm{\acute{e}t}}\). Then
\begin{enumerate}
\item For \(i > n = \dim X\), the cohomology group
  \[
    H^i((V,N), (u_{X,*} \mathbb{L}) \otimes \mathcal{O}_{X_{\log}}^{+}/p)
  \]
  is almost zero as \(\mathcal{O}_K\)-module.
\item Assume \(V' \subset V\) is a rational subset which is strictly contained
  in \(V\). Then the image of
  \[
    H^i((V,N), (u_{X,*} \mathbb{L}) \otimes \mathcal{O}_{X_{\log}}^{+}/p) \to H^i((V',N'=V'
    \times_V N), (u_{X,*} \mathbb{L}) \otimes \mathcal{O}_{X_{\log}}^{+}/p)
  \]
  is an almost finitely generated \(\mathcal{O}_K\)-module.
\end{enumerate}
\end{lemma}

\begin{proof}
  This follows from the proof of~\cite[Lemma 5.6]{Sch13}. In the argument we
  need to replace ~\cite[Lemma 4.5]{Sch13} by \Cref{lemmafet} and \Cref{lemmaloget},~\cite[Lemma 4.12]{Sch13} by \Cref{analogue Lemma
    4.12},~\cite[Lemma 5.3]{Sch13} by~\Cref{analogue Lemma 5.3} and~\cite[Lemma
  5.5]{Sch13} by~\Cref{analogue Lemma 5.5}.
\end{proof}

\begin{lemma}
\label{analogue Lemma 5.8}
Let \(K\) be a perfectoid field of characteristic \(0\) containing all
\(p\)-power roots of unity. Let \(\mathbb{L}\) be an \(\mathbb{F}_p\)-local
system on \(U_{\mathrm{\acute{e}t}}\). Then
\[
  H^j(X_{\log},(u_{X *}\mathbb{L}) \otimes \mathcal{O}^+_X/p)
\]
is an almost finitely generated \(\mathcal{O}_K\)-module, which is almost zero
for \(j > 2 \dim X\).
\end{lemma}

\begin{proof}
  Consider the projection \(\mu \colon X_{\log} \to X_{an}\) sending \(U\) to
  \((U,U)\). Previous \Cref{analogue Lemma 5.6} shows that \(R^j \lambda_*(u_{X
    *}\mathbb{L} \otimes \mathcal{O}^+_X/p)\) is almost zero for \(j > \dim X\).
  Notice that any covering of \((X,X)\) in $X_{\log}$ can be refined by ones meeting
  the condition of previous Lemma. The cohomological dimension of \(X_{an}\) is
  \(\leq \dim X\) by~\cite[Proposition 2.5.8]{dJvdP}, we get the desired
  vanishing result. The proof above is similar to the counterpart of~\cite[Lemma
  5.8]{Sch13}.

  The proof of almost finitely generatedness is also similar to that
  in~\cite[Lemma 5.8]{Sch13}. Again, we have to replace~\cite[Lemma 5.3]{Sch13}
  by \Cref{analogue Lemma 5.3} and~\cite[Lemma 5.6]{Sch13} by \Cref{analogue
    Lemma 5.6}.
\end{proof}

\begin{definition}
Let $(X,D)$ be as before. The tilted integral structure sheaf $\hat{\mathcal{O}}^+_{X^{\flat}_{\log}}$ is given by $\varprojlim \mathcal{O}^{+}_{X_{\log}}/p$ where the inverse limit is taken along the Frobenius map. 
Set $\hat{\mathcal{O}}_{X^{\flat}_{\log}}=\hat{\mathcal{O}}^+_{X^{\flat}_{\log}}[\frac{1}{p}]$.
\end{definition}

The next lemma follows from repeating the argument of its untilted version (Lemma \ref{lemmascholze410}).
\begin{lemma}\label{lemmascholze510}
Let $K$ be a perfectoid field of characteristic $0$, and let $X$ be an adic space associated to a rigid space over $\mathrm{Sp}(K)$. 
Let $N\in X_{\mathrm{prolog}}$ be affinoid perfectoid, with $\hat{N}=\mathrm{Spa}(R,R^+)$ where $(R,R^+)$ is a perfectoid affinoid $(K,K^o)$-algebra. Let $(R^{\flat},R^{\flat+})$ be its tilt. Then we have
\begin{enumerate}
\item $\hat{\mathcal{O}}^+_{X^{\flat}_{\log}}(N)=R^{\flat+}$ and $\hat{\mathcal{O}}_{X^{\flat}_{\log}}(N)=R^{\flat}$;
\item The cohomology groups $H^i(N,\hat{\mathcal{O}}_{X^{\flat}_{\log}}^+)$ are almost zero for $i>0$, with respect to the almost setting defined by the maxiaml ideal of topologically nilpotent elements in $K^{\circ}$.
\end{enumerate}
\end{lemma}

Now we can follow Scholze's method to show \Cref{thmpc}.

\begin{proof}[Proof of \Cref{thmpc}] 
To simplify our notations, throughout the proof, we still denote $\nu^*(u_{X,*}(\mathbb{L}))$ by $u_{X,*}(\mathbb{L})$. Note that $K^{\flat}$ is an algebraically closed field of characteristic $p$. Fix an element $\pi\in \mathcal{O}_{K^b}$ such that $(\pi)^{\sharp} =p$.
Note that $\hat{\mathcal{O}}^+_{X^{\flat}_{\log}}$ is a sheaf of perfect flat $\mathcal{O}_{K^{\flat}}$-algebras with $\hat{\mathcal{O}}^+_{X^{\flat}_{\log}}/{\pi^k}=\mathcal{O}^{+}_{X}/p^k$ (by Lemma \ref{lemmascholze510} and Lemma \ref{lemmascholze410}).
Let $M_k=H^i(X_{\mathrm{prolog}}, u_{X,*}(\mathbb{L})\otimes \hat{\mathcal{O}}^+_{X^{\flat}_{\log}}/{\pi^k})^a$. 
It follows from Lemma \ref{analogue Lemma 5.8} that $M_k$ satisfy the hypotheses of \cite[Lemma 2.12]{Sch13}. Hence there is some $r \in \mathbb{N}$ such that $M_k=(K^{\flat o}/\pi^k)^r$ as almost $K^{\flat}$-modules, compatibly with the Frobenius action. 
By \Cref{perfectoid basis}, Lemma \ref{lemmascholze510} and \cite[Lemma 3.18]{Sch13}, we have $$R\varprojlim (u_{X,*}(\mathbb{L})\otimes \hat{\mathcal{O}}^+_{X^{\flat}_{\log}}/{\pi^k})^a= (u_{X,*}(\mathbb{L})\otimes \hat{\mathcal{O}}^+_{X^{\flat}_{\log}})^a.$$
Therefore, we have
\[H^i(X_{\mathrm{prolog}}, u_{X,*}(\mathbb{L})\otimes \hat{\mathcal{O}}^+_{X^{\flat}_{\log}})^a \cong \varprojlim H^i(X_{\mathrm{prolog}}, u_{X,*}(\mathbb{L})\otimes \hat{\mathcal{O}}^+_{X^{\flat}_{\log}}/ {\pi^k})^a \cong (\mathcal{O}_{K^{\flat}}^a)^r.\]
Note that the site $X_{\mathrm{prolog}}$ is algebraic and the final object $(X,X) \in X_{\mathrm{prolog}}$ is coherent. We invert $\pi$ and get 
\[H^i(X_{\mathrm{prolog}}, u_{X,*}(\mathbb{L})\otimes \hat{\mathcal{O}}_{X^{\flat}_{\log}}) \cong (K^{\flat})^r\]
which is still compatible with the action of $\mathrm{Frob}$. Then we use the Artin-Schreier sequence
\[
0\rightarrow u_{X,*}(\mathbb{L})\rightarrow u_{X,*}(\mathbb{L})\otimes \hat{\mathcal{O}}_{X^{\flat}_{\log}} \xrightarrow{h} u_{X,*}(\mathbb{L})\otimes \hat{\mathcal{O}}_{X^{\flat}_{\log}}\rightarrow 0
\]
where the map $h$ sends $v\otimes f$ to $v\otimes (f^p-f)$. 
This is an exact sequence of sheaves: by Lemma \ref{lemmext}, $u_{X,*}(\mathbb{L})$ is locally coming from a $\mathbb{F}_p$-local system on $X_{\mathrm{\acute{e}t}}$, moreover, $u_{X,*}(\mathbb{F}_p)=\mathbb{F}_p$ on $X_{\log}$. 
Therefore, it suffices to check the map $h$ is surjective locally on affinoid perfectoid $N\in X_{\mathrm{prolog}}$ and over which $u_{X,*}(\mathbb{L})$ is trivial. 
Note that $\hat{N}^{\flat}_{\mathrm{F}\mathrm{\acute{e}t}}\cong \hat{N}_{\mathrm{F}\mathrm{\acute{e}t}}$, and finite \'{e}tale covers of $\hat{N}$ come via pullback from finite \'{e}tale covers in $X_{\mathrm{prolog}}$ by \cite[Lemma 7.5 (i)]{Scholze1}. 

Denote $X_{\mathrm{prolog}}$ by $X$. The Artin-Schreier sequence gives 
\[\xymatrix{ \dots H^i(X,u_{X,*}(\mathbb{L}))\ar[r] \ar@{=}[d] & H^i(X, u_{X,*}(\mathbb{L})\otimes \hat{\mathcal{O}}_{X^{\flat}_{\log}})\ar[r]\ar@{=}[d]  & H^i(X, u_{X,*}(\mathbb{L})\otimes \hat{\mathcal{O}}_{X^{\flat}_{\log}}) \ar@{=}[d] \ldots \\
\mathbb{F}_p^r \ar[r] & (K^{\flat})^r \ar[r] & (K^{\flat})^r} \]
where the map $(K^{\flat})^r \rightarrow (K^{\flat})^r$ is coordinate-wise $x\mapsto x^p-x$. The map $(K^{\flat})^r \rightarrow (K^{\flat})^r$ is surjective since $K^{\flat}$ is algebraically closed. Using Lemma \ref{lemmacohomology} (2), we have
\[ 
H^i(X_{\log},u_{X,*}(\mathbb{L}))= H^i(X_{\mathrm{prolog}},u_{X,*}(\mathbb{L}))=H^i(X_{\mathrm{prolog}},u_{X,*}(\mathbb{L}) \otimes \hat{\mathcal{O}}_{X^{\flat}_{\log}})^{\mathrm{Frob}=id}=\mathbb{F}^r_p.
\]
which implies the theorem.
\end{proof}

\begin{remark}
\label{rmkpc}
By the same proof, one has the following variant of \Cref{thmpc}:
  let $X$ be a proper smooth rigid analytic space over $\mathrm{Sp}(k)$ with an
  SSNC divisor $D$. Let $\mathbb{L}$ be an $\mathbb{F}_p$-local system on
  $(X-D)_{\mathrm{\acute{e}t}}$. Then there is an isomorphism of almost
  finitely generated $\hat{\bar{k}}^{\circ a}$-modules
  \[
    H^{i}((X,X_{\bar{k}}),\nu^*(u_{X,*}(\mathbb{L}))) \otimes_{\mathbb{F}_p} \hat{\bar{k}}^{\circ a}/p \cong
    H^i((X,X_{\bar{k}}), \nu^*(u_{X *}(\mathbb{L})) \otimes \mathcal{O}^+_{X_{\log}}/p)
  \]
  for all $i \geq 0$. 
  Here $X_{\bar{k}}$ is the pro-system of $X_l$ where $l/k$ runs through all finite extension of $k$,
see also~\cite[Proposition 3.15]{Sch13} and the discussion before it.
\end{remark}

\section{The period sheaves}
\label{section 6}

\begin{definition}
On $X_{\mathrm{log}}$ we have the \emph{sheaf of log differentials} 
\[
\Omega^1_{X_{\mathrm{log}}}(\log D) \coloneqq \lambda^{-1}(\Omega^1_X(\log D))
\bigotimes_{\lambda^{-1}(\mathcal{O}_X)} \mathcal{O}_{X_{\mathrm{log}}}
\]
where $\lambda \colon X_{\mathrm{log}} \to X_{\mathrm{\acute{e}t}}$ is the natural map sending 
$(V \to X)$ to $(V,V)$. Note that this is a locally finite free sheaf of $\mathcal{O}_{X_{\mathrm{log}}}$-modules.
\end{definition}

The following definitions are similar to \cite[Definition 6.1]{Sch13}. 
\begin{definition} 
Let $X$ be a rigid space over $\mathrm{Sp}(k)$ with SSNC divisor $D$. We have the following sheaves on $X_{\mathrm{prolog}}$.
\begin{enumerate}
\item The sheaf $\mathbb{A}_{\mathrm{inf}}:=W(\hat{\mathcal{O}}_{X^{\flat}_{\log}}^+)$ 
and its rational version $\mathbb{B}_{\mathrm{inf}}:=\mathbb{A}_{\mathrm{inf}}[\frac{1}{p}]$. 
We have $\theta:\mathbb{A}_{\mathrm{inf}}\rightarrow \hat{\mathcal{O}}_{X_{\log}}^+$ 
extended to $\theta:\mathbb{B}_{\mathrm{inf}}\rightarrow \hat{\mathcal{O}}_{X_{\log}}$.
\item The positive de Rham sheaf is given by 
$\mathbb{B}_{\mathrm{dR}}^+:=\varprojlim \mathbb{B}_{\mathrm{inf}}/(\ker \theta)^n$ with its filtration 
$\mathrm{Fil}^i\mathbb{B}_{\mathrm{dR}}^+=(\ker \theta)^i\mathbb{B}_{\mathrm{dR}}^+$.
\item The de Rham sheaf $\mathbb{B}_{\mathrm{dR}}=\mathbb{B}_{\mathrm{dR}}^+[t^{-1}],$ where 
$t$ is any element that generates $\mathrm{Fil}^1\mathbb{B}_{\mathrm{dR}}^+$. 
It has the filtration $\mathrm{Fil}^i\mathbb{B}_{\mathrm{dR}}= \sum \limits_{j \in \mathbb{Z}}t^{-j} \mathrm{Fil}^{i+j}\mathbb{B}_{\mathrm{dR}}^+$.
\end{enumerate}
\end{definition}

The analogue of~\cite[6.2-6.7]{Sch13} holds in our setting with the same proof, let us summarize it in the following:

\begin{remark}
\label{summary remark}
Let $K$ be a perfectoid field which is the completion of some algebraic extension of $k$
and fix $\pi \in K^{\flat}$ such that $\pi^{\sharp}/p \in (K^{\circ})^\times$. 
Let $(V,N)$ be an affinoid perfectoid in the localized site 
$X_{\mathrm{prolog}}/\mathrm{Spa}(K,K^{\circ})$ 
with $\hat{N} = \mathrm{Spa}(R,R^+)$. Then we have
\begin{enumerate}
\item There is an element $\xi \in \mathbb{A}_{\mathrm{inf}}(K,K^{\circ})$ that generates 
$\ker (\theta \colon \mathbb{A}_{\mathrm{inf}} (R,R^+) \to R^+) $, and is not a zero-divisor in $\mathbb{A}_{\mathrm{inf}}(R,R^+)$.
\item we have a canonical isomorphism
\[
\mathbb{A}_{\mathrm{inf}}(V,N) = \mathbb{A}_{\mathrm{inf}}(R,R^+),
\]
and analogous statements hold for $\mathbb{B}_{\mathrm{inf}}$, $\mathbb{B}_{\mathrm{dR}}^+$ 
and $\mathbb{B}_{\mathrm{dR}}$.
In particular, $\mathrm{Fil}^1 \mathbb{B}_{\mathrm{dR}}^+(V,N)$ is a principal ideal in 
$\mathbb{B}_{\mathrm{dR}}^+$ generated by a non-zero-divisor $\xi \in \mathbb{A}_{\mathrm{inf}}(K,K^{\circ})$.
\item All $H^i((V,N),\mathcal{F})$ are almost zero for $i > 0$, where $\mathcal{F}$ is any of the sheaves above. In particular, 
\[
\mathrm{gr}^{\bullet} \mathbb{B}_{\mathrm{dR}} (V,N) = \mathrm{gr}^{\bullet} \mathbb{B}_{\mathrm{dR}} (R,R^+) =
R[\xi^{\pm 1}].
\]
\item Let $S$ be a profinite set, and let $(V,N') = (V,N \times S) \in X_{\mathrm{prolog}}$
which is again affinoid perfectoid. Then
\[
\mathcal{F}(V,N') = \mathrm{Hom}_{cont}(S,\mathcal{F}(V,N))
\]
for any of the sheaves
\[
\mathcal{F} \in \{\hat{\mathcal{O}}_{X_{\log}},\hat{\mathcal{O}}_{X_{\log}}^+,\hat{\mathcal{O}}_{X_{\log}^{\flat}},\hat{\mathcal{O}}_{X_{\log}^{\flat}}^+,\mathbb{A}_{\mathrm{inf}},\mathbb{B}_{\mathrm{inf}},
\mathbb{B}_{\mathrm{dR}}^+,\mathbb{B}_{\mathrm{dR}},\mathrm{gr}^i\mathbb{B}_{\mathrm{dR}} \}.
\]
\end{enumerate}

For all $i \in \mathbb{Z}$, we have $\mathrm{gr}^i\mathbb{B}_{\mathrm{dR}} \cong \hat{\mathcal{O}}_{X_{\log}}(i)$
as sheaves on $X_{\mathrm{prolog}}$ where $(i)$ denotes a Tate twist in the same sense as
in~\cite[Proposition 6.7]{Sch13}.
\end{remark}

\begin{definition}
On $X_{\mathrm{log}}$ we have the \emph{sheaf of log differentials} 
\[
\Omega^1_{X_{\mathrm{log}}}(\log D) \coloneqq \lambda^{-1}(\Omega^1_X(\log D))
\bigotimes_{\lambda^{-1}(\mathcal{O}_X)} \mathcal{O}_{X_{\mathrm{log}}}
\]
where $\lambda \colon X_{\mathrm{log}} \to X_{\mathrm{\acute{e}t}}$ is the natural morphism of sites sending 
$(V \to X)$ to $(V,V)$. Note that this is a locally finite free sheaf of
$\mathcal{O}_{X_{\mathrm{log}}}$-modules.
\end{definition}

\begin{remark}
Note that the $(V,N)$'s in $X_{\log}$ satisfying the following conditions:
\begin{enumerate}
\item $V$ (hence $N$) is an affinoid space;
\item there is an \'{e}tale morphism $V \to \mathbb{T}^{n-r,r}(\underline{Z})$ such that 
\[
g^{-1}(D) = \bigcup_{l=n-r+1}^{n} V(Z_l)
\]
where $g \colon V \to X$ is the structure map;
\item there is a finite \'{e}tale morphism $N \to V[\sqrt[m]{Z_l}]$;
\end{enumerate}
form a basis of $X_{\log}$ by \Cref{Abhyankar's Lemma} and \Cref{analogue Lemma 5.2}.
For $(V,N)$ satisfying the above conditions with $N = \mathrm{Sp}(R)$, we have an isomorphism
\[
\Omega^1_{X_{\log}}(\log D)(V,N) \cong \bigoplus_{1 \leq l \leq n} R\cdot \frac{dZ_l}{Z_l}.
\]
Hence for such a $(V,N)$, we have $\Omega^1_{X_{\log}}(\log D)(V,N) = \Omega^1_{N}(\log (f^{-1}D))(N)$. 
Here $f \colon N \to X$ is induced from $N \to V \to X$.
\end{remark}

\begin{definition}
\label{period sheaves}
Let $X$ be a smooth rigid adic space over $\mathrm{Sp}(k)$ where $k$ is a discretely valued
complete non-archimedean extension of $\mathbb{Q}_p$ with perfect residue field $\kappa$.
Consider the following sheaves on $X_{\mathrm{prolog}}$.

\begin{enumerate}
\item The sheaf of differentials 
\[
\Omega^1_{X}(\log D) \coloneqq \nu^* \bigg( \Omega^1_{X_{\mathrm{log}}}(\log D) \bigg).
\footnote{One should notice the difference between $\nu^{-1}$ and $\nu^*$.}
\] 
We also define $\Omega^i_{X}(\log D):=\wedge^i\Omega^1_{X}(\log D)$.
\item The positive logarithmic structural de Rham sheaf $\mathcal{O}\mathbb{B}_{\log\mathrm{dR}}^+$ 
is given by the sheafification of the presheaf sending affinoid perfectoid $(V,N)$ with
\[
N=\varprojlim N_i=\varprojlim \mathrm{Sp}(R_i) \textnormal{~and~}\hat{N}=\mathrm{Sp}a(R,R^+)
\] 
to the colimit over $i$ of
\[
\label{outcome} 
\tag{\epsdice{1}}
\varprojlim_r \frac{\bigg(\big(R^{\circ}_i \hat{\otimes}_{W(\kappa)}
(\mathbb{A}_{\mathrm{inf}}(R,R^+)/\ker(\theta)^r)
\big)[\frac{1 \otimes [f_k^{\flat}]}{f_k \otimes 1} ] [\frac{1}{p}]\bigg)}{\ker(1 \otimes \theta)^r}
\]
Here $\{f_k \in \mathcal{O}(V)\}$ are defining functions of $D_k$ given as part of the
definition of $(V,N)$ being affinoid perfectoid. 
The completed tensor product is the $p$-adic completion of the tensor product. Here 
$1 \otimes \theta$ is the tensor product of the map $R^{\circ}_i \rightarrow R^+$ and 
$\theta \colon \mathbb{A}_{\mathrm{inf}}(R,R^+) \rightarrow R^+$,
moreover it sends $\frac{1 \otimes [f_k^{\flat}]}{f_k \otimes 1}$ to $1$.
Note that $R$ contains all roots of $f_k$, therefore we have $f_k^{\flat}=
(f_k, (f_k)^{\frac{1}{p}},\ldots,) \in (R^+)^\flat$, in particular $\theta ([f_k^{\flat}])=f_k$. 

\item The uncompleted logarithmic structure de Rham sheaf is given by
$\mathcal{O}\mathbb{B}_{\log\mathrm{dR}}^{uc} \coloneqq \mathcal{O}\mathbb{B}_{\log\mathrm{dR}}^+[t^{-1}]$ 
where $t$ is a generator of $\mathrm{Fil}^1 \mathbb{B}_{\mathrm{dR}}^+$.
\end{enumerate}
It is clear that we still have the map $\theta:\mathcal{O}\mathbb{B}_{\log\mathrm{dR}}^+ \rightarrow
\hat{\mathcal{O}}^+_X$ which induces its filtration \[
\mathrm{Fil}^i\mathcal{O}\mathbb{B}_{\log\mathrm{dR}}^+=(\ker \theta)^i \mathcal{O}\mathbb{B}_{\log\mathrm{dR}}^+.
\] 
We also have a filtration on $\mathcal{O}\mathbb{B}_{\log\mathrm{dR}}^{uc}$ by
\[
\mathrm{Fil}^i\mathcal{O}\mathbb{B}_{\log\mathrm{dR}}^{uc}=\sum_{j\in \mathbb{Z}} t^{-j}
\mathrm{Fil}^{i+j}\mathcal{O}\mathbb{B}_{\log\mathrm{dR}}^+.
\]
(4) Finally, the logarithmic structure de Rham sheaf is defined to be the completion of uncompleted logarithmic structure de Rham sheaf with respect to the filtration defined above\footnote{We thank Xinwen Zhu for pointing out to us that the original sheaf we defined was not complete, and we need to take completion with respect to this filtration, c.f.~\cite[Remark 3.11]{DLLZ}.}
\[
\mathcal{O}\mathbb{B}_{\log\mathrm{dR}} \coloneqq \widehat{\mathcal{O}\mathbb{B}_{\log\mathrm{dR}}^{uc}}.
\]
Note that $\mathcal{O}\mathbb{B}_{\log\mathrm{dR}}$ is equipped with the filtration coming from that on $\mathcal{O}\mathbb{B}_{\log\mathrm{dR}}^{uc}$, with respect to which it is complete, and that both two sheaves have the same graded pieces.
\end{definition}

\begin{remark}
(1) It is easy to check that the colimit over $i$ of \ref{outcome}
does not depend on the presentation of $N$, and it does define a presheaf.

(2) Later on we will see that for a set of basis $(V,N)$ of $X_{\mathrm{prolog}}$, there is a cofinal system
of $i$'s such that the outcomes \ref{outcome} corresponding to $i$ are the same, see~\Cref{compute OBDR}.

(3) Note that we have a natural $\mathbb{B}_{\mathrm{dR}}^+$-linear connection with log poles:
\[
\mathcal{O}\mathbb{B}_{\log\mathrm{dR}}^+ \xrightarrow{\nabla} \mathcal{O}\mathbb{B}_{\log\mathrm{dR}}^+ 
\otimes_{\mathcal{O}_{X_{\log}}} \Omega^1_{X_{\log}}(\log D)
\]
sending 
\[
\frac{1 \otimes [f_k^{\flat}]}{f_k \otimes 1} \mapsto 
-\frac{1 \otimes [f_k^{\flat}]}{(f_k \otimes 1)^2} \,df_k \\
= -\frac{1 \otimes [f_k^{\flat}]}{f_k \otimes 1} \cdot \,d\log(f_k),
\]
extended from the connection $\mathcal{O}_{X_{\log}} \xrightarrow{\nabla}
\Omega^1_{X_{\log}}(\log D)$. Because $t \in \mathbb{B}_{\mathrm{dR}}^+$,
inverting it, we get a natural $\mathbb{B}_{\mathrm{dR}}$-linear connection with
log poles:
\[
  \mathcal{O}\mathbb{B}_{\log\mathrm{dR}}^{uc} \xrightarrow{\nabla}
  \mathcal{O}\mathbb{B}_{\log\mathrm{dR}}^{uc} \otimes_{\mathcal{O}_{X_{\log}}}
  \Omega^1_{X_{\log}}(\log D).
\]
Take completion with respect to the induced filtration, we get:
\[
  \mathcal{O}\mathbb{B}_{\log\mathrm{dR}} \xrightarrow{\nabla}
  \mathcal{O}\mathbb{B}_{\log\mathrm{dR}} \otimes_{\mathcal{O}_{X_{\log}}}
  \Omega^1_{X_{\log}}(\log D).
\]

(4) The definition of these de Rham period sheaves uses the fact that $X$ is
defined over a $p$-adic field. This is the crucial place where we have to use
this fact.\footnote{ We thank Bhargav Bhatt for reminding us this in a private
  communication.}
\end{remark}

We describe $\mathcal{O}\mathbb{B}_{\log\mathrm{dR}}^+$ in the following
proposition (see also~\cite[Proposition 6.10]{Sch13} and~\cite{Corrigendum}).
Let $U \subset X$ be an open. Let $K$ be a perfectoid field which is the
completion of an algebraic extension of $k$. We get the base change $U_K$ of $U$
to $\mathrm{Sp}(K)$, and again consider $U_K \in X_{\mathrm{prolog}}$ by slight
abuse of notation. Let $\varphi : U \rightarrow
\mathbb{T}^{n-r,r}(\underline{Z})$ (cf.~\Cref{model example}) be an \'{e}tale
morphism such that $f_k \coloneqq \varphi^*(Z_{n-r+k})$ $(k=1,\ldots,r)$ defines
the component $D_k$ of $D \cap U$. Note that such $U$'s form a basis of $X$. Let
$\tilde{U}=U \times_{\mathbb{T}^{n-r,r}} \tilde{\mathbb{T}}^{n-r,r}$. Taking a
further base change to $K$, we get $(U_K,\tilde{U}_K) \in X_{\mathrm{prolog}}$
is perfectoid.

\begin{proposition} 
\label{compute OBDR}
Let notations be as above. Consider the localized site
$X_{\mathrm{prolog}}/(U_K,\tilde{U}_K)$. We have the elements
\[
  u_i=Z_i \otimes 1-1\otimes [Z_i^{\flat}] \in
  \mathcal{O}\mathbb{B}_{\log\mathrm{dR}}^+|_{(U,\tilde{U} ) }
\]
for $i=1,\ldots, n-r$, and
\[
  u_j=1-\frac{1 \otimes [Z_j^{\flat}]}{Z_j \otimes 1}\in
  \mathcal{O}\mathbb{B}_{\log\mathrm{dR}}^+|_{(U,\tilde{U})}
\]
for $j=n-r+1,\ldots, n$. Here we abuse the notations by using $Z_j$ to denote
$\varphi^*(Z_j)=f_j$. We will also use $Z_j$ (resp.~$[Z_j^{\flat}]$) to denote
$Z_j \otimes 1$ (resp.~$[Z_j^{\flat}] \otimes 1$) to simplify our notations.

The map
\[
  \mathbb{B}_{\mathrm{dR}}^+|_{(U_K,\tilde{U}_K)}[\![X_1,\ldots,X_n]\!]
  \rightarrow \mathcal{O}\mathbb{B}_{\log\mathrm{dR}}^+|_{(U_K,\tilde{U}_K)}
\]
sending $X_i$ to $u_i$ is an isomorphism of sheaves over
$X_{\mathrm{prolog}}/(U_K,\tilde{U}_K)$.
\end{proposition}

\begin{proof}
\textbf{Step 0}: definition of the map. 

Let $(V,N)$ be an affinoid perfectoid over $(U_K,\tilde{U}_K)$ where $N =
\varprojlim N_i$ with $N_i = \mathrm{Spa}(R_i,R_i^{\circ})$ and $\hat{N} =
\mathrm{Spa}(R,R^+)$. For each $r$ and $i$, we use the fact that
\[
\frac{\mathbb{B}_{\mathrm{dR}}^+(R,R^+)[\![X_1,\ldots,X_n]\!]}{(\xi,X_i)^r} \cong
\frac{\mathbb{A}_{inf}(R,R^+)[\![X_1,\ldots,X_n]\!][1/p]}{(\xi,X_i)^r}
\]
to define the morphism
\[
  \frac{\mathbb{A}_{inf}(R,R^+)[\![X_1,\ldots,X_n]\!][1/p]}{(\xi,X_i)^r} \to
  \frac{\bigg(\big(R^{\circ}_i \hat{\otimes}_{W(\kappa)}
    (\mathbb{A}_{\mathrm{inf}}(R,R^+)/\ker(\theta)^r) \big)[\frac{1 \otimes
      [f_k^{\flat}]}{f_k \otimes 1} ] [\frac{1}{p}]\bigg)}{\ker(1 \otimes
    \theta)^r} \eqqcolon S_{i,r},
\]
by sending any element $a \in \mathbb{A}_{inf}(R,R^+)$ to $1 \otimes a$ and
$X_i$ to $u_i$ as described in the statement of this proposition. Here we used
the fact that the ideal $(\xi,X_i)$ is sent inside $\ker(1 \otimes \theta)$.
Taking inverse limit over $r$ and then colimit over $i$ gives the morphism in
the statement of this proposition.

We want to show that for any $N_i$ there exists a higher $N_{i'} \to N_i$ such
that the morphism
\[
\label{key morphism}
\tag{\epsdice{2}}
\frac{\mathbb{A}_{inf}(R,R^+)[\![X_1,\ldots,X_n]\!][1/p]}{(\xi,X_i)^r} \to
S_{i',r}
\]
is an isomorphism for all $r$. This shows in particular that in \Cref{period
  sheaves}(2), there is a cofinal system of $i$'s for which the outcomes
(\ref{outcome}) are the same.

\textbf{Step 1}: construct a section.

Let $i$ be large enough, so that we get a log \'{e}tale morphism $(V,N_i) \to
\mathbb{T}^{n-r,r}$ where $N_i = \mathrm{Spa}(R_i,R_i^{\circ})$. By
\Cref{Abhyankar's Lemma}, we see that there is an $m \in \mathbb{N}$ such that
$(N_i \times_{\mathbb{T}^{n-r,r}} \mathbb{T}^{n-r,r}[\sqrt[m]{Z_l}])^{\nu}
\eqqcolon \mathrm{Spa}(R_{i'},R_{i'}^{\circ}) \to
\mathbb{T}^{n-r,r}[\sqrt[m]{Z_l}]$ is \'{e}tale. We will take
$\mathrm{Spa}(R_{i'},R_{i'}^{\circ})$ to be the $N_{i'}$ we want.

To simplify the notations further, let us denote $\mathbb{B}_r \coloneqq
\frac{\mathbb{B}_{\mathrm{dR}}^+(R,R^+)[\![X_1,\ldots,X_n]\!]}{(\xi,X_i)^r}$.
For technical reason we also want to consider, for each $r$, the
$\mathbb{B}_{\mathrm{dR}}^+(R,R^+)$-algebra $\mathbb{B}'_r \coloneqq
\frac{\mathbb{B}_{\mathrm{dR}}^+(R,R^+)[X_1,\ldots,X_{n-r},\tilde{X}_{n-r+1},\ldots,\tilde{X}_n]}
{(\xi,X_1,\ldots,X_{n-r},\tilde{X}_{n-r+1},\ldots,\tilde{X}_n)^r}$. There is a
natural morphism $\mathbb{B}'_r \xrightarrow{\beta_r} \mathbb{B}_r$ where
$\beta_r(\tilde{X}_l) = \frac{[(Z_l^{1/m})^{\flat}]}{(1-X_l)^{1/m}} -
[(Z_l^{1/m})^{\flat}]$. Note that $\frac{1}{(1-X_l)^{1/m}}$ can be written as a
power series in $\mathbb{Q}[\![X_l]\!]$, hence our expression makes sense. We
still denote the composition $\theta \circ \beta_r$ by $\theta$.

In the following, we will show that there is a natural morphism $R_{i'} \to
\mathbb{B}'_r$, whose image is contained in a open and bounded (w.r.t.~the
$p$-adic topology induced from $\mathbb{B}_r$) subring inside $\mathbb{B}'_r$,
which is compatible with $\theta$ map for all $r$.

First note that for all $r$, there is a map
\[
W(\kappa)[p^{-1}][Z_1^{\pm 1},\ldots,Z_{n-r}^{\pm 1},Z_{n-r+1}^{1/m},\ldots,Z_n^{1/m}] 
\to \mathbb{B}'_r
\]
by sending $Z_j \mapsto X_j+[Z_j^\flat]$ for $j \leq n-r$ and
$Z_l^{1/m} \mapsto \tilde{X}_l$ for all $l > n-r$.

Now we need the following lemma.

\begin{lemma}
  Let $\mathcal{O}$ be an excellent complete rank $1$ valuation ring with a
  pseudo-uniformizer $\varpi$, and let $F$ be its fraction field which is viewed
  as a non-archimedean field. Let $A_0^+$ be a finitely presented flat
  $\mathcal{O}$-algebra. Let $A = \widehat{A_0^+}[1/p]$, where the completion is
  with respect to $\varpi A_0^+$, which is an affinoid $F$-algebra. Let $U =
  \mathrm{Sp}(B)$ be an affinoid rigid space admitting an \'{e}tale map $U \to
  \mathrm{Sp}(A)$. Then there exists a finitely presented $\mathcal{O}$-flat
  $A_0^+$-algebra $B_0^+$, such that $B_0 = B_0^+[1/p]$ is \'{e}tale over
  $A_0^+[1/p]$ and $B^{\circ}$ is the $\varpi$-adic completion of $B_0^+$.
\end{lemma}


\begin{proof}
This is a slight generalization of~\cite[Lemma 6.12]{Sch13} and it
follows from the same proof as~\cite[Proof of Lemma 6.12]{Sch13}.
\end{proof}

Apply the above lemma to $\mathcal{O} = W(\kappa)$, $A_0^+ = W(\kappa)[Z_1^{\pm
  1},\ldots,Z_{n-r}^{\pm 1},Z_{n-r+1}^{1/m},\ldots,Z_n^{1/m}]$ and $B = R_{i'}$
gives a finitely generated $W(\kappa)[Z_1^{\pm 1},\ldots,Z_{n-r}^{\pm
  1},Z_{n-r+1}^{1/m},\ldots,Z_n^{1/m}]$-algebra $R_{i'0}^{\circ}$ whose generic
fibre $R_{i'0}$ is \'{e}tale over
\[
  W(\kappa)[p^{-1}][Z_1^{\pm 1},\ldots,Z_{n-r}^{\pm
    1},Z_{n-r+1}^{1/m},\ldots,Z_n^{1/m}].
\]
By Hensel's Lemma, we get a unique lift $R_{i'0} \to \mathbb{B}'_r$. In
particular we get a lift of $R_{i'0}^{\circ}$. This extends to the $p$-adic
completion with image lands in an open bounded subring (see~\cite[Lemma 6.11
and its proof]{Sch13}). Hence we get a lift of $R_{i'} \to \mathbb{B}'_r$ with
image lands in an open bounded subring.

\textbf{Step 2}: injectivity of~\ref{key morphism}.

After composing with $\beta_r$,
we get a map (recall that $\frac{[Z_l^\flat]}{Z_l} = 1 - X_l$)
\[
S_{i',r} \to \mathbb{B}_r
\]
for which the composition
\[
\mathbb{B}_r \to S_{i',r}
\to \mathbb{B}_r
\]
is the identity. Therefore we see that~\ref{key morphism} is injective.

\textbf{Step 3}: surjectivity of~\ref{key morphism}.

Now we only need to show that
\[
\mathbb{B}_r \to S_{i',r}
\]
is surjective. Let us consider the following commutative diagrams
\[
  \xymatrix{ \mathbb{B}'_r \ar[d]^{\beta_r} \ar[r]^(0.20){\alpha_r} &
    \frac{(R_{i'}^{\circ} \hat{\otimes}_{W(\kappa)}
      (\mathbb{A}_{\mathrm{inf}}(R,R^+)/\ker(\theta)^r))[\frac{1}{p}]}{\ker (1
      \otimes \theta)^r}
    \ar[d]^{\epsilon_r} \\
    \mathbb{B}_r
    \ar[r]^(0.35){\gamma_r} & S_{i',r} \\
  }
\]
and
\[
  \xymatrix{ \frac{(R_{i'}^{\circ} \hat{\otimes}_{W(\kappa)}
      (\mathbb{A}_{\mathrm{inf}}(R,R^+)/\ker(\theta)^r))[\frac{1}{p}]}{\ker (1
      \otimes \theta)^r}
    \ar[r]^(0.70){\epsilon_r} \ar[d] & S_{i',r} \\
    \frac{(R_{i'}^{\circ} \hat{\otimes}_{W(\kappa)}
      (\mathbb{A}_{\mathrm{inf}}(R,R^+)/\ker(\theta)^r))[\frac{1}{p}]}{\ker (1
      \otimes \theta)^r} [Y_{n-r+1},\ldots,Y_n] \ar@{->>}[ur]^{\delta_r} }
\]
where $\alpha_r(\tilde{X}_l) = Z_l^{1/m} \otimes 1 - 1 \otimes
[(Z_l^{1/m})^{\flat}]$, $\delta_r(Y_l) = \frac{[Z_l^\flat]}{Z_l}$ for all $l >
n-r$ and $\epsilon_r$ is the natural morphism. Note that $\delta_r$ is a
surjection. Also the formula $\gamma_r(1-X_l) = \frac{[Z_l^\flat]}{Z_l}$ tells
us that $\frac{[Z_l^\flat]}{Z_l}$ is in the image of $\gamma_r$. Therefore to
show $\gamma_r$ is surjective, it suffices to show that $\alpha_r$ is
surjective. This just follows from the argument in~\cite{Corrigendum} and is
written down below for the sake of completeness of our argument.

First the map
\[
\label{map}
\tag{\epsdice{3}} \xymatrix{
  \frac{R_{i'}^{\circ}[X_1,\ldots,X_{n-r},\tilde{X}_{n-r+1},\ldots,\tilde{X}_n]}
  {(X_1,\ldots,X_{n-r},\tilde{X}_{n-r+1},\ldots,\tilde{X}_n)^r} \ar[r] &
  (R_{i'}^{\circ} \hat{\otimes}_{W(\kappa)} R_{i'}^{\circ})/(\ker \theta_{i'})^r \\
}
\]
is injective, with cokernel killed by a power of p, where $\theta_{i'} \colon
R_{i'}^{\circ} \hat{\otimes}_{W(\kappa)} R_{i'}^{\circ} \to R_{i'}^{\circ}$ is
the multiplication map. Here we used the fact that $\mathrm{Sp}(R_{i'}) \to
\mathbb{T}^{n-r,r}[\sqrt[m]{Z_l}]$ is \'{e}tale.

Recall that we have constructed, in step 1, a map $R_{i'}^{\circ} \to
\mathbb{B}'_r$ taking values in some open and bounded subring. Composing with
the projection onto $\mathbb{B}_{\mathrm{dR}}^+/\ker(\theta)^r$, we see that
there is a map $R_{i'} \to \mathbb{B}_{\mathrm{dR}}^+/\ker(\theta)^r$ compatible
with $\theta$ taking values in some open and bounded subring $\mathbb{B}_{r,0}
\subset \mathbb{B}_{\mathrm{dR}}^+/\ker(\theta)^r$ (notice the typo
in~\cite{Corrigendum} here). Now we apply $\hat{\otimes}_{R_{i'}^{\circ}}
\mathbb{B}_{r,0}$ to the map (\ref{map}). We get
\[
\mathbb{B}_{r,0}[X_1,\ldots,X_{n-r},\tilde{X}_{n-r+1},\ldots,\tilde{X}_n]/
(X_1,\ldots,X_{n-r},\tilde{X}_{n-r+1},\ldots,\tilde{X}_n)^r \to
\]
\[
(R_{i'}^{\circ} \hat{\otimes}_{W(\kappa)} \mathbb{B}_{r,0})/((\ker \theta_{i'})^r
\hat{\otimes}_{R_{i'}^{\circ}} \mathbb{B}_{r,0})
\]
is an isomorphism up to a bounded power of $p$. Finally we invert $p$ and use
\[
  ((\ker \theta_{i'})^r \hat{\otimes}_{R_{i'}^{\circ}} \mathbb{B}_{r,0}) \subset
  (\ker \theta)^r
\]
to conclude that $\alpha_r$ is a surjection.
\end{proof}

\begin{corollary}[logarithmic Poincar\'{e} Lemma]
\label{corologPL}
Let $X$ be a smooth rigid space of dimension $n$ over $\mathrm{Sp}(k)$ with SSNC
divisor $D$. The following sequence of sheaves on $X_{\mathrm{prolog}}$ is
exact.
\[
  0 \rightarrow \mathbb{B}_{\mathrm{dR}}^+ \rightarrow
  \mathcal{O}\mathbb{B}_{\log\mathrm{dR}}^+ \xrightarrow{\nabla}
  \mathcal{O}\mathbb{B}_{\log\mathrm{dR}}^+ \otimes_{\mathcal{O}_X}
  \Omega_X^1(\log D) \xrightarrow{\nabla} \ldots \xrightarrow{\nabla}
  \mathcal{O}\mathbb{B}_{\log\mathrm{dR}}^+\otimes_{\mathcal{O}_X}
  \Omega_X^n(\log D) \rightarrow 0.
\]
Moreover, the derivation $\nabla$ satisfies Griffiths transversality with
respect to the filtration on $\mathcal{O}\mathbb{B}_{\log\mathrm{dR}}^+$, and
with respect to the grading giving $\Omega^i_X(\log D)$ degree $i$, the sequence
is strict exact.
\end{corollary}

\begin{proof}
This follows from \Cref{compute OBDR} and the equation
\[
  \,d (X_l) = \,d (1-\frac{[Z_l^\flat]}{Z_l}) = \frac{[Z_l^\flat]}{Z_l^2} \,d
  Z_l = \frac{[Z_l^\flat]}{Z_l} \frac{\,d Z_l}{Z_l} = (1-X_l) \cdot \,d
  \log(Z_l).
\]
\end{proof}

\begin{remark}
  From the above Corollary, especially the strict exactness, we get the following exact
  sequence
\[
\label{Poincare}
0 \rightarrow \mathbb{B}_{\mathrm{dR}} \rightarrow
\mathcal{O}\mathbb{B}_{\log\mathrm{dR}} \xrightarrow{\nabla}
\mathcal{O}\mathbb{B}_{\log\mathrm{dR}} \otimes_{\mathcal{O}_X} \Omega_X^1(\log
D) \xrightarrow{\nabla} \ldots \xrightarrow{\nabla}
\mathcal{O}\mathbb{B}_{\log\mathrm{dR}} \otimes_{\mathcal{O}_X} \Omega_X^n(\log
D) \rightarrow 0
\]
which share the same properties as the sequence above.
\end{remark}

In particular, we get the following short exact sequence, which is due to
Faltings in the case of algebraic varieties, see~\cite[Theorem 4.3]{Fal}.

\begin{corollary}[Faltings's extension]
\label{Faltings's extension}
Let $X$ be a smooth rigid space over $\mathrm{Sp}(k)$ with SSNC divisor $D$.
Then we have a short exact sequence of sheaves over $X_{\mathrm{prolog}}$,
\[
  0 \rightarrow \hat{\mathcal{O}}_{X_{\log}}(1) \rightarrow \mathrm{gr}^1
  \mathcal{O} \mathbb{B}_{\log\mathrm{dR}}^+ \rightarrow
  \hat{\mathcal{O}}_{X_{\log}} \otimes_{\mathcal{O}_{X_{\log}}} \Omega_X^1(\log
  D) \rightarrow 0
\]
\end{corollary}

\begin{corollary}
\label{analogue Corollary 6.15}
Let $X \to \mathbb{T}^{n-r,r}, \tilde{X}, K$ and $X_i$ be as above. For any $i \in \mathbb{Z}$,
we have an isomorphism of sheaves over $X_{\mathrm{prolog}}/(X_K, \tilde{X}_K)$,
\[
\mathrm{gr}^i \mathcal{O}\mathbb{B}_{\log\mathrm{dR}} \cong \xi^i\hat{\mathcal{O}}_{X_{\log}}[X_1/\xi,\ldots,X_n/\xi].
\]
In particular,
\[
\mathrm{gr}^{\bullet}\mathcal{O}\mathbb{B}_{\log\mathrm{dR}} \cong 
\hat{\mathcal{O}}_{X_{\log}}[\xi^{\pm 1},X_1,\ldots,X_n],
\]
where $\xi$ and $X_i$ have degree $1$.
\end{corollary}

The following is analogous to~\cite[Proposition 6.16]{Sch13}.

\begin{proposition}
\label{analogue Proposition 6.16}
Let $X = \mathrm{Spa}(R,R^{\circ})$ be an affinoid adic space of finite type over
$\mathrm{Spa}(k,k^{\circ})$ with an \'{e}tale map $X \to \mathbb{T}^{n-r,r}$ that factors
as a composite of rational embeddings and finite \'{e}tale maps.
\begin{enumerate}
    \item Assume that $K$ contains all roots of unity. Then
    \[
    H^q (X_{K,\mathrm{prolog}}, \mathrm{gr}^0 \mathcal{O}\mathbb{B}_{\log\mathrm{dR}}) = 0
    \]
    unless $q=0$, in which case it is $R \hat{\otimes}_k K$.
    \item We have
    \[
    H^q (X_{\mathrm{prolog}}, \mathrm{gr}^i \mathcal{O}\mathbb{B}_{\log\mathrm{dR}}) = 0
    \]
    unless $i = 0$ and $q = 0,1$. We also have $H^0 (X_{\mathrm{prolog}},\mathrm{gr}^0
    \mathcal{O}\mathbb{B}_{\log\mathrm{dR}}) = R$ and $H^1 (X_{\mathrm{prolog}},\mathrm{gr}^0
    \mathcal{O}\mathbb{B}_{\log\mathrm{dR}}) = R \log \chi$. Here $\chi \colon \mathrm{Gal}(\bar{k}/k)
    \to \mathbb{Z}_p^\times$ is the cyclotomic character and
    \[
    \log \chi \in \mathrm{Hom}_{cont}(\mathrm{Gal}(\bar{k}/k),\mathbb{Q}_p) = 
    H^1_{cont}(\mathrm{Gal}(\bar{k}/k),\mathbb{Q}_p)
    \]
    is its logarithm.
\end{enumerate}
\end{proposition}

\begin{proof}
(1) As before, denote $X_K \times_{\mathbb{T}^{n-r,r}_K} \tilde{\mathbb{T}}^{n-r,r}_K 
\eqqcolon \tilde{X}_K$ where $\widehat{\tilde{X}_K} = \mathrm{Spa}(\hat{R},\hat{R}^{\circ})$. 
We see that $\tilde{X}_K \to X_K$ is a $\mathbb{Z}_p^{n-r} \times
\hat{\mathbb{Z}}^r$-cover and all multiple-fold fibre products of $\tilde{X}_K$ over
$X_K$ are affinoid perfectoid. By \Cref{analogue Corollary 6.15} and \Cref{summary remark}
we see that all higher cohomology groups of the sheaves considered vanish and
\[
H^q (X_{K,\mathrm{prolog}}, \mathrm{gr}^0\mathcal{O}\mathbb{B}_{\log\mathrm{dR}}) = 
H^q_{cont} (\mathbb{Z}_p^{n-r} \times \hat{\mathbb{Z}}^r,
\mathrm{gr}^0\mathcal{O}\mathbb{B}_{\log\mathrm{dR}}(\tilde{X}_K)).
\]
Note that we may write
\[
\mathrm{gr}^0\mathcal{O}\mathbb{B}_{\log\mathrm{dR}}(\tilde{X}_K) = \hat{R} [V_1,\ldots,V_n],
\]
where $V_i = t^{-1} \log(\frac{[T_i^\flat]}{T_i})$ and $t = \log ([\epsilon])$.
Let $\gamma_i$ be the $i$-th basis vector of $\mathbb{Z}_p^{n-r} \times
\hat{\mathbb{Z}}^r$, then we have (c.f.~\cite[Lemma 6.17]{Sch13})
\[
\gamma_i(V_j) = V_j + \delta_{ij}.
\]

Next we claim the inclusion
\[
(R \hat{\otimes}_k K)[V_1,\ldots,V_n] \subset \hat{R}[V_1,\ldots,V_n]
\]
induces an isomorphism on the continuous group cohomologies. This can be seen via checking the
graded pieces given by the degree of polynomials. On the gradeds the group action on $V_i$'s
is trivial, therefore it suffices to check that $R \hat{\otimes}_k K \subset \hat{R}$
induces an isomorphism on continuous group cohomologies. This just follows from
\Cref{lemmafet}(2) and \Cref{analogue Lemma 5.5}, c.f.~\cite[Lemma 6.18]{Sch13}.

Lastly we need to compute
\[
H^q_{cont} (\mathbb{Z}_p^{n-r} \times \hat{\mathbb{Z}}^r, (R \hat{\otimes}_k K)[V_1,\ldots,V_n]).
\]
But since all the factors $\hat{\mathbb{Z}}_{(p)} \coloneqq \prod_{l \not= p} \mathbb{Z}_l$
acts trivially on $(R \hat{\otimes}_k K)[V_1,\ldots,V_n]$ which has $p$-adic topology, we see that
the continuous group cohomology is the same as
\[
H^q_{cont} (\mathbb{Z}_p^n, (R \hat{\otimes}_k K)[V_1,\ldots,V_n]).
\]
Now the last paragraph of the proof of~\cite[Proposition 6.16(i)]{Sch13} shows that these
cohomology groups are $0$ whenever $q > 0$ and is equal to $R \hat{\otimes}_k K$ when $q=0$.

(2) Let $k'$ be the completion of $\cup_{p \nmid n} k(\mu_n)$ and take $K$ as the completion of
$k'(\mu_{p^{\infty}})$. Also let us denote $G = \mathrm{Gal}(k(\mu_{\infty})/k) = 
H \times \Gamma$ where $H = \mathrm{Gal}((\cup_{p \nmid n} k(\mu_n))/k)$ and $\Gamma = 
\mathrm{Gal}(k(\mu_{p^\infty})/k)$. By the same argument as in the proof
of~\cite[Proposition 6.16(ii)]{Sch13}, we see that
\[
H^q (X_{\mathrm{prolog}}, \mathrm{gr}^i \mathcal{O}\mathbb{B}_{\log\mathrm{dR}}) = 
H^q_{cont}(G, R \hat{\otimes}_k K(i))
\]
and
\[
H^q_{cont}(\Gamma, R \hat{\otimes}_k K(i)) = R_{k'} \otimes_{\mathbb{Q}_p} H^q_{cont} (\Gamma, \mathbb{Q}_p(i))
\]
and the latter is well-known, see~\cite{p-divisible}.
Moreover we know that the action of $H$ on $\log \chi$ is trivial and
\[
H^q_{cont}(H, R_{k'}) = 0
\]
unless $q=0$ in which case it is $R$. Indeed, since $H$ is a profinite group, we know that 
$H^q_{cont}(H, R_{k'}) = (H^q_{cont}(H, R^{\circ} \hat{\otimes}_{\mathcal{O}_k} \mathcal{O}_{k'}))[1/p]$.
Now it suffices to show $H^q_{cont}(H, R^{\circ} \hat{\otimes}_{\mathcal{O}_k} \mathcal{O}_{k'})=0$ for all $q > 0$,
and $H^0_{cont}(H, R^{\circ} \hat{\otimes}_{\mathcal{O}_k} \mathcal{O}_{k'})= R^{\circ}$. 
We claim that $H^q_{cont}(H, (R^{\circ} \hat{\otimes}_{\mathcal{O}_k} \mathcal{O}_{k'})/\varpi^m) = 0$ 
for all $m > 0$, unless $q=0$ in which case it is given by $R^{\circ}/\varpi^m$. To prove this claim we simply notice
that by induction on $m$ and the fact that $R^{\circ} \hat{\otimes}_{\mathcal{O}_k} \mathcal{O}_{k'}$ is
$\varpi$-torsion free, it suffices to prove it when $m=1$ which follows from Hilbert 90. The above claim yields that
$H^q_{cont}(H, R^{\circ} \hat{\otimes}_{\mathcal{O}_k} \mathcal{O}_{k'}) = R^q \varprojlim_m R^{\circ}/\varpi^m$,
which easily implies what we want.

Put all these together along with Hochschild--Serre spectral sequence yields the results we want.
\end{proof}

\begin{corollary}
\label{pushforward obdr}
Let $X$ be a smooth adic space over $\mathrm{Spa}(k,\mathcal{O}_{k})$ with an SSNC divisor $D$.
Let $i,j$ be two integers and let $m$ be a positive integer, then we have
\begin{enumerate}
\item $R^q \nu_* (\mathrm{Fil}^{i}\mathcal{O}\mathbb{B}_{\log\mathrm{dR}}/\mathrm{Fil}^{i+m}\mathcal{O}\mathbb{B}_{\log\mathrm{dR}})
= 0$ unless $q = 0,1$ and $0 \in [i,i+m)$, in which case $R^0 \nu_*$ is given by $\mathcal{O}_{X_{\log}}$ and 
$R^1 \nu_*$ is given by $\mathcal{O}_{X_{\log}} \cdot \log \chi$.
\item $R^q \nu_* \mathrm{Fil}^{i}\mathcal{O}\mathbb{B}_{\log\mathrm{dR}} = 0$ unless $q = 0,1$ and $i \leq 0$ in which case
$R^0 \nu_*$ is given by $\mathcal{O}_{X_{\log}}$ and  $R^1 \nu_*$ is given by $\mathcal{O}_{X_{\log}} \log \chi$.
The above computation also holds for $i = -\infty$ where $\mathrm{Fil}^{-\infty}\mathcal{O}\mathbb{B}_{\log\mathrm{dR}} = 
\mathcal{O}\mathbb{B}_{\log\mathrm{dR}}$.
\item $R^i \nu_* \hat{\mathcal{O}}_{X_{\log}}(j) = 0$ unless 
\begin{itemize}
    \item $i = j$ in which case it is given by $\Omega^j_{X_{\log}}(\log D)$ or;
    \item $i = j+1$ in which case it is given by $\Omega^j_{X_{\log}}(\log D) \cdot \log \chi$.\footnote{Note the typo
    in~\cite[Remark 6.20]{Sch13}.}
\end{itemize}
Moreover the isomorphism $R^1 \nu_* \hat{\mathcal{O}}_{X_{\log}}(1) \cong \Omega^1_{X_{\log}}(\log D)$ is
given by the Faltings's extension (c.f.~\Cref{Faltings's extension}). 
\end{enumerate}
\end{corollary}

\begin{proof}
(1) trivially follows from \Cref{analogue Proposition 6.16}(2).

(2) follows from (1) by commuting limit and colimit with cohomology.

(3) follows from applying $R \nu_*$ to $j$-th graded piece of~\Cref{Poincare} which reads
\[
0 \to \hat{\mathcal{O}}_{X_{\log}}(j) \to \mathrm{gr}^j \mathcal{O}\mathbb{B}_{\log\mathrm{dR}} \to
\mathrm{gr}^{j-1} \mathcal{O}\mathbb{B}_{\log\mathrm{dR}} \otimes_{\mathcal{O}_{X_{\log}}} \Omega^1_{X_{\log}}(\log D)
\to \cdots.
\]
The last statement can be seen via the natural morphism from the sequence in \Cref{Faltings's extension} to the 
above sequence where $j=1$.
\end{proof}

\begin{remark}
\label{ramified cover obdr}
Let $X=\mathrm{Sp}(R) \xrightarrow{f} \mathbb{T}^{n-r,r}$ be as in~\Cref{analogue Proposition 6.16}.
Denote $f^*(T_l)$ by $f_l$ where $l > n-r$. Then by the same argument, one can show that
\[
H^q ((X,X[\sqrt[m]{f_l}]_{\bar{k}}),\mathrm{gr}^0 \mathcal{O}\mathbb{B}_{\log\mathrm{dR}}) = 0
\]
unless $q=0$, in which case it is $R[\sqrt[m]{f_l}] \hat{\otimes}_k \hat{\bar{k}}$.
\end{remark}

\begin{remark}
Let $X$ be a smooth adic space over $\mathrm{Spa}(C,\mathcal{O}_C)$ where $C$ is an algebraically closed
non-archimedean extension of $\mathbb{Q}_p$. Similar as in~\cite[Proposition 3.23 and Lemma 3.24]{ScholzeCDM},
one can show that there is a commutative diagram
\[
\xymatrix{
\bigwedge^k (R^1 \nu_* \hat{\mathcal{O}}_{X_{\log}}(1)) \ar[r] \ar[d]^{\cong} & 
R^k \nu_* \hat{\mathcal{O}}_{X_{\log}}(k) \ar[d]^{\cong} \\
\bigwedge^k (\Omega^1_{X_{\log}}(\log D)) \ar[r] &
\Omega^k_{X_{\log}}(\log D)
}
\]
where the vertical maps are obtained in the same fashion as above.
\end{remark}

\section{Comparisons}
\label{section 7}

\subsection{Vector bundles on $X_{\log}$}
\label{subsection 7.1}

\begin{definition}
A \emph{vector bundle} $\mathcal{F}$ on $X_{\mathrm{log}}$ is a sheaf of $\mathcal{O}_{X_{\log}}$-modules
such that there exists a finite affinoid covering 
$(V_i,N_i) \to (X,X)$
and finite projective $\Gamma (N_i,\mathcal{O}_{N_i})$-modules $M_i$ with isomorphism
\[
\mathcal{F}|_{(V_i,N_i)} \cong M_i \otimes_{\Gamma (N_i,\mathcal{O}_{N_i})} \mathcal{O}_{X_{\log}}.
\]
Here $(M_i \otimes_{\Gamma (N_i,\mathcal{O}_{N_i})} \mathcal{O}_{X_{\log}}) (W,M) \coloneqq 
M_i \otimes_{\Gamma (N_i,\mathcal{O}_{N_i})} \Gamma(M, \mathcal{O}_{M})$ 
for any object $(W,M)$ over $(V_i,N_i)$, and by affinoid covering we mean a covering
with all $V_i$ (hence $N_i$) being affinoid.
\end{definition}

\begin{remark}
Note that since $M_i$'s are assumed to be finite projective, they are direct summand in finite free
modules. Therefore $M_i \otimes_{\mathcal{O}(N)} \mathcal{O}_{X_{\log}}$ indeed defines a sheaf
on the localized site $X_{\log}/(V,N)$. We say $\mathcal{F}$ is represented by a finite projective
module on $(V,N) \in X_{\log}$ if one can find an $M$ and an isomorphism as in the previous definition.
\end{remark}

\begin{theorem}[Theorem A]
\label{Theorem A}
Let $\mathcal{F}$ be a vector bundle on $X_{\log}$. Then there exists a positive integer $m$ such that
for any affinoid $V \xrightarrow{f} X$ \'{e}tale over $X$ with $f^{-1}(D_l)$ being defined by $f_l$ (where $D_l$ is the $l$-th component of $D$), 
there exists a finite projective $\mathcal{O}(V[\sqrt[m]{f_l}])$-module $M$ and an isomorphism
\[
\mathcal{F}|_{(V,V[\sqrt[m]{f_l}])} \cong M \otimes_{\mathcal{O}(V[\sqrt[m]{f_l}])} \mathcal{O}_{X_{\log}}.
\]
\end{theorem}

\begin{proof}
Let $V_i$ and $N_i$ be as in the definition, by passing to refinement we may assume the preimage of $D_l$
in $V_i$ is defined by a single function $f_{i.l}$. Then by \Cref{Abhyankar's Lemma} we can find a positive
integer $m$ such that $N_i[\sqrt[m]{f_{i,l}}] \to V_i[\sqrt[m]{f_{i,l}}]$ 
is finite \'{e}tale. Consider the following
diagram:
\[
\xymatrix{
\coprod_i (V_i',N_i') \ar@{=}[r] & \coprod_i \bigg( (V_i,N_i) \times_{(X,X)} (V,V[\sqrt[m]{f_l}]) \bigg)
\ar[r] \ar[d] & (V,V[\sqrt[m]{f_l}]) \ar[d] \\
& \coprod_i (V_i,N_i) \ar[r] & (X,X). \\
}
\]
From the diagram and our choice of $m$, 
we see that $\coprod_i N_i' \to V[\sqrt[m]{f_l}]$ is an \'{e}tale covering in the 
usual sense in rigid geometry and our sheaf $\mathcal{F}$ is represented by finite projective
modules $M_i$ on $(V_i',N_i')$. Therefore \'{e}tale descent implies what we want.
\end{proof}

\begin{theorem}[Theorem B]
\label{Theorem B}
For any vector bundle $\mathcal{F}$ and any affinoid $(V,N) \in X_{\log}$, assume one of the following
conditions holds
\begin{enumerate}
\item $\mathcal{F}|_{(V,N)}$ is represented by a finite projective $\mathcal{O}(N)$-module $M$ or;
\item preimage of $D_l$ in $V$ is defined by a single function $f_l$ for all $l$,
\end{enumerate}
then we have
\[
H^q((V,N),\mathcal{F}) = 0
\]
for all $q > 0$.
\end{theorem}

\begin{proof}
We first observe that the statement of this theorem for objects satisfying condition (2) implies
the statement for objects satisfying condition (1). Indeed, we can cover $V$ by $V_i$ satisfying
(2). Therefore by the statement for objects satisfying condition (2), we see that 
$(V_i,V_i \times_V N) \to (V,N)$ is an acyclic cover for $\mathcal{F}$. Hence
by \v{C}ech-to-cohomology spectral sequence we see that
$H^q((V,N),\mathcal{F})$ is the same as $q$-th \v{C}ech cohomology for this covering, which is the cohomology of 
\v{C}ech complex associated to the affinoid covering $\{V_i \times_V N\}$ for our finite projective
module $M$. Hence we get $H^q((V,N),\mathcal{F}) = 0$ as $N$ is an affinoid. 

From now on we will assume that our $(V,N)$ satisfies condition (2). 
We will prove the vanishing of
cohomology by induction on $q$ (the starting case $q=1$ follows from the same argument),
therefore we will assume for objects satisfying (2) the cohomology of 
$\mathcal{F}$ vanishes up to degree $q-1$.

Let $\xi \in H^q((V,N),\mathcal{F})$ be a cohomology class. Then there exists a covering by qcqs objects
$(V',N') \to (V,N)$ such that $\xi$ pulls back to zero in $H^q((V',N'),\mathcal{F})$.
Then by \Cref{Abhyankar's Lemma} and \Cref{Theorem A} we can find an $m$ such that 
\begin{enumerate}
\item $\mathcal{F}|_{(V,N[\sqrt[m]{f_l}])}$ is represented by a finite projective
$\mathcal{O}(N[\sqrt[m]{f_l}])$-module $M$;
\item $N'' \eqqcolon (N[\sqrt[m]{f_l}] \times_N N')^{\nu} \to N[\sqrt[m]{f_l}]$ is an \'{e}tale covering.
\end{enumerate}
Let $k' = k[\zeta_m]$ where $\zeta_m$ is a primitive $m$-th root of unity.
Let us consider the following diagram
\[
\xymatrix{
(V,N[\sqrt[m]{f_l}]_{k'}) \ar[d]^{\alpha} & (V',N''_{k'}) \ar[l]_(0.4){\beta} \ar[d] \\
(V,N) & (V',N') \ar[l] \\
}
\]
where subscript $(\cdot)_{k'}$ means the base change of spaces from $k$ to $k'$.
The cohomology class $\xi$ is assumed to be zero on $(V',N')$, hence it is zero on $(V',N''_{k'})$.
Now by \v{C}ech-to-cohomology spectral sequence
\[
\xymatrix{ E^{a,b}_2=\textnormal{\v{H}}^a(\beta, \underline{H}^b{\mathcal{F}}) \ar@{=>}[r]
& H^{a+b}((V,N[\sqrt[m]{f_l}]_{k'}),\mathcal{F})}
\]
and induction hypotheses, we have an exact sequence as follows
\[
0 \rightarrow \textnormal{\v{H}}^{q}(\beta, \mathcal{F}) \rightarrow
H^q((V,N[\sqrt[m]{f_l}]_{k'}),\mathcal{F}) \rightarrow 
\textnormal{H}^{q}((V',N''_{k'}), \mathcal{F}).
\]
From this sequence, we see that $\xi_{(V,N[\sqrt[m]{f_l}]_{k'})}$
is represented by a class in \v{C}ech cohomology of $\mathcal{F}$
associated to the cover given by $\beta$. Moreover,
for $q\geq 1$, we have that $\textnormal{\v{H}}^{q}(\beta, \mathcal{F})$ is zero by 
(1), (2) and \'etale descent. It follows that $\xi_{(V,N[\sqrt[m]{f_l}]_{k'})}=0$.

Therefore, as above, by induction hypothese
and \v{C}ech-to-cohomology spectral sequence 
we see that $\xi$ is represented by a class in $\check{H}^q(\alpha,\mathcal{F})$, the
\v{C}ech cohomology of $\mathcal{F}$
associated to the cover given by $\alpha$. Now we notice that the $j$-th fold product of 
$(V,N[\sqrt[m]{f_l}]_{k'})$ over $(V,N)$ is isomorphic to 
$(V,N[\sqrt[m]{f_l}]_{k'}) \times G \times \ldots \times G$ with $(j-1)$-st copies of $G$'s appearing in
the product. Here $G$ is the Galois group of the covering $N[\sqrt[m]{f_l}]_{k'}$ over $N$.
The sheaf condition gives us an action of $G$ on $M$ and 
$\check{H}^q(\alpha,\mathcal{F}) = H^q(G,M)$ which is zero because $M$ is divisible
and $G$ is a finite group. This proves that
$\xi = 0$. 
\end{proof}

\begin{corollary}
\label{pushforward vector bundle}
Let $\lambda \colon X_{\log} \to X_{\mathrm{\acute{e}t}}$ be the morphism of sites sending $U$ to $(U,U)$.
Then we have
\[
R \lambda_{*} \mathcal{O}_{X_{\log}} \cong \mathcal{O}_{X_{\mathrm{\acute{e}t}}}
\]
Therefore for any vector bundle $\mathcal{F}$ on $X_{\mathrm{\acute{e}t}}$, we have
\[
\mathcal{F} \cong R \lambda_{*} \lambda^* \mathcal{F}.
\]
In particular, we see that $\lambda^*(\cdot)$ gives a fully faithful embedding from the category of vector bundles
on $X_{\mathrm{\acute{e}t}}$ to that on $X_{\log}$.
\end{corollary}

\begin{proof}
The first assertion follows from Theorem A and B above. The second assertion follows from adjunction formula.
\end{proof}

\begin{theorem}
\label{Theorem C}
For any vector bundle $\mathcal{F}$ on $X_{\log}$, the cohomology groups
\[
H^q((X,X),\mathcal{F})
\]
are finite dimensional $k$ vector spaces for all $q$.
\end{theorem}

\begin{proof}
By \Cref{analogue Lemma 5.3} we may find two affinoid coverings $\{V_i\}$ and $\{V_i'\}$ of $X$,
such that
\begin{enumerate}
\item $V_i' \Subset_X V_i$ for all $i$;
\item $V_i$ (hence $V_i'$) satisfies condition (2) in Theorem B, i.e., $D_l \cap V_i$ is given by
vanishing of $f_{i,l}$.
\end{enumerate}
Now by Theorem A, there is an $m$ such that $\mathcal{F}|_{(V_i,V_i[\sqrt[m]{f_{i,l}}])}$ is represented
by a finite projective module $M_i$. By the same reasoning $\mathcal{F}|_{(V_i',V_i'[\sqrt[m]{f_{i,l}}])}$
is represented by $M_i|_{V_i'[\sqrt[m]{f_{i,l}}]}$. By Theorem B, we see that the covering
$\coprod_i (V_i,V_i[\sqrt[m]{f_{i,l}}]) \to (X,X)$ 
(resp.~$\coprod_i (V_i',V_i'[\sqrt[m]{f_{i,l}}]) \to (X,X)$) is acyclic for $\mathcal{F}$.
Therefore we see that
\[
H^q((X,X),\mathcal{F}) = \check{H}^q(\coprod_i (V_i,V_i[\sqrt[m]{f_{i,l}}]) \to (X,X), \mathcal{F})
\]
\[
= \check{H}^q(\coprod_i (V_i',V_i'[\sqrt[m]{f_{i,l}}]) \to (X,X), \mathcal{F}).
\]
On the other hand, by our choice of $V_i$ and $V_i'$, we have that $V_i'[\sqrt[m]{f_{i,l}}]$ is
strictly contained in $V_i[\sqrt[m]{f_{i,l}}]$. Therefore the map from the \v{C}ech complex of
$(V_i,V_i[\sqrt[m]{f_{i,l}}])$ to that of $(V_i',V_i'[\sqrt[m]{f_{i,l}}])$ is strictly continuous
and an isomorphism on cohomology groups. Hence we see that these cohomology groups are finite dimensional
$k$ vector spaces. See also the proof of Kiehl's proper mapping theorem 
in~\cite[6.4]{Lectures}.
\end{proof}

The above theorem implies the following base change lemma, which will be used later.

\begin{lemma}
\label{analogue Lemma 7.13}
Let $X$ be a smooth adic space over $\mathrm{Spa}(k,\mathcal{O}_{k})$ with an SSNC divisor $D$.
Let $\mathcal{A}$ be a vector bundle on $X_{\log}$. Then for all $i,j \in \mathbb{Z}$, 
we have an isomorphism
\[
H^j((X,X),\mathcal{A}) \otimes_k \mathrm{gr}^i B_{\mathrm{dR}} \cong
H^j((X,X_{\bar{k}}),\mathcal{A} \otimes_{\mathcal{O}_{X_{\log}}} \mathrm{gr}^i 
\mathcal{O}\mathbb{B}_{\mathrm{logdR}})
\]
where the latter group is computed on $X_{\mathrm{prolog}}$.
\end{lemma}

\begin{proof}
By twisting, it suffices to prove the case where $i=0$. The statement reads
\[
H^j((X,X),\mathcal{A}) \otimes_k \hat{\bar{k}} \cong
H^j((X,X_{\bar{k}}),\mathcal{A} \otimes_{\mathcal{O}_{X_{\log}}} \mathrm{gr}^0 
\mathcal{O}\mathbb{B}_{\mathrm{logdR}}).
\]
To this end, let $\coprod (V_i,V_i[\sqrt[m]{f_{i,l}}])$ be an acyclic covering of
$\mathcal{A}$ as in the proof of~\Cref{Theorem C}. Denote the \v{C}ech complex
associated to $\mathcal{A}$ and this covering by $\mathcal{C}^{\bullet}$.
By~\Cref{ramified cover obdr}, we know that RHS is cohomology groups of
$\mathcal{C}^{\bullet} \hat{\otimes}_k \hat{\bar{k}}$.
Therefore we reduce to the statement
\[
H^j(\mathcal{C}^{\bullet}) \hat{\otimes}_k \hat{\bar{k}} \cong
H^j(\mathcal{C}^{\bullet} \hat{\otimes}_k \hat{\bar{k}}).
\]
This follows from the fact that $\mathcal{C}^{\bullet}$ has finite dimensional (as $k$ vector spaces)
cohomology groups.
\end{proof}

\begin{remark}
(1) There are interesting vector bundles on $X_{\log}$ not coming from $X_{\mathrm{\acute{e}t}}$.
Assume $D \subset X$ is a smooth divisor, the ``square root'' of the ideal sheaf of $D$,
given by $\sqrt{I_D}(V,N)=\{a\in \Gamma(N,\mathcal{O}_{N})|a^2 \in g^* I_{(D)}\}$
(for $g \colon N \to V$), is
such an example.

(2) One can develop a more general theory of ``coherent'' sheaf and prove similar theorems
as above for these sheaves. We will not work it out in this note however, since it is
irrelevant to the theme of this note.
\end{remark}

\subsection{Proof of the Comparison}
\label{subsection 7.2}

In this subsection, 
let $k$ be an discretely valued complete non-archimedean extension of 
$\mathbb{Q}_p$ with perfect residue field $\kappa$. 
Let $X$ be a smooth adic space over $\mathrm{Spa}(k,\mathcal{O}_{k})$
with an SSNC divisor $D$. Denote $X \setminus D$ by $U$.
Denote an algebraic closure of $k$ by $\overline{k}$ and its completion by $\hat{\bar{k}}$. 
Let $A_{\mathrm{inf}}$, $B_{\mathrm{inf}}$, etc.~be the period rings as defined by Fontaine.

\begin{theorem}
\label{first comparison}
There is a canonical isomorphism
\[
H^{m}\big((X,X_{\bar{k}}),\mathbb{B}_{\mathrm{dR}}^+ \big) 
\otimes_{B_{\mathrm{dR}}^+} B_{\mathrm{dR}} \cong 
H^m\big(X, \Omega_{X_{\log}}^{\bullet}(\log D) \big) \otimes_k B_{\mathrm{dR}}
\]
compatible with filtrations and $\mathrm{Gal}(\bar{k}/k)$-actions.

Moreover, we have a $\mathrm{Gal}(\bar{k}/k)$-equivariant isomorphism
\[
H^m\big((X,X_{\bar{k}}), \hat{\mathcal{O}}_{X_{\log}} \big) \cong 
\bigoplus_{a+b=m} H^{a}(X,\Omega_{X}^{b}(\log D)) \otimes_k \hat{\bar{k}}(-b).
\]
\end{theorem}

\begin{remark}
By~\Cref{pushforward vector bundle}, we have canonical isomorphisms:
\[
H^m\big(X, \Omega_{X_{\log}}^{\bullet}(\log D) \big) \cong
H^m\big((X,X), \Omega_{X_{\log}}^{\bullet}(\log D) \big)
\]
and
\[
H^{a}(X,\Omega_{X}^{b}(\log D)) \cong H^{a}((X,X),\Omega_{X}^{b}(\log D)),
\]
where the left hand side denotes the cohomology computed on the rigid space $X$ and
the right hand side denotes the cohomology computed on the Faltings site $X_{\log}$.
\end{remark}

\begin{proof}
In the filtered derived category we have
\[
R\Gamma\big((X,X_{\bar{k}}),\mathbb{B}_{\mathrm{dR}}^+\big) 
\otimes_{B_{\mathrm{dR}}^+} B_{\mathrm{dR}} = 
R\Gamma\big((X,X_{\bar{k}}),\mathbb{B}_{\mathrm{dR}} \big)
= R\Gamma\big((X,X_{\bar{k}}), \mathcal{O}\mathbb{B}_{\mathrm{logdR}}
\otimes_{\mathcal{O}_{X_{\log}}} \Omega_{X_{\log}}^{\bullet}(\log D) \big)
\]
where the second equality follows from Poincar\'{e} Lemma (c.f.~\Cref{Poincare}).
We claim that the natural map of filtered complexes
\[
\Omega_{X_{\log}}^{\bullet}(\log D) \to 
\mathcal{O}\mathbb{B}_{\mathrm{logdR}}
\otimes_{\mathcal{O}_{X_{\log}}} \Omega_{X_{\log}}^{\bullet}(\log D)
\]
induces a quasi-isomorphism
\[
R\Gamma\big((X,X), \Omega_{X_{\log}}^{\bullet}(\log D) \big) \otimes_k
\mathrm{B}_{\mathrm{dR}} \to 
R\Gamma\big((X,X_{\bar{k}}), \mathcal{O}\mathbb{B}_{\mathrm{logdR}}
\otimes_{\mathcal{O}_{X_{\log}}} \Omega_{X_{\log}}^{\bullet}(\log D) \big).
\]
It suffices to check the claim above on graded pieces. 
Further filtering by using naive filtration of 
$\Omega_{X_{\log}}^{\bullet}(\log D)$, 
one is reduced to show that for any vector bundle $\mathcal{A}$
on $X_{\log}$ and $i \in \mathbb{Z}$, the map
\[
R\Gamma\big((X,X), \mathcal{A} \big) \otimes_k
\mathrm{gr}^i \mathrm{B}_{\mathrm{dR}} \to 
R\Gamma\big((X,X_{\bar{k}}), \mathcal{A} \otimes_{\mathcal{O}_{X_{\log}}}
\mathrm{gr}^i \mathcal{O}\mathbb{B}_{\mathrm{logdR}}\big)
\]
is a quasi-isomorphism.
This follows from~\Cref{analogue Lemma 7.13}. 

Therefore we have constructed a quasi-isomorphism
\[
R\Gamma\big((X,X), \Omega_{X_{\log}}^{\bullet}(\log D) \big) \otimes_k
\mathrm{B}_{\mathrm{dR}} \to
R\Gamma\big((X,X_{\bar{k}}),\mathbb{B}_{\mathrm{dR}}^+ \big) 
\otimes_{B_{\mathrm{dR}}^+} B_{\mathrm{dR}}
\]
in filtered derived category. Now we get comparison results, by taking cohomology
of both sides (resp.~of the $0$-th graded piece of both sides).
\end{proof}

Let us make a remark about the notion of local systems on sites $X_{\log}$
and $X_{\mathrm{prolog}}$.

\begin{remark}
\label{limit remark}
Note that for any $\mathbb{Z}/p^n$-local system $\mathbb{L}_n$ on $U$, 
$(u_{X,*}\mathbb{L}_n)(V,N)= \mathbb{L}_n(N^{\circ})$ for any $(V,N) \in X_{\log}$ (see \Cref{comparison logtopos}). 
By \Cref{lemmacohomology}(1), for any $(V,N = \varprojlim N_i) \in X_{\mathrm{prolog}}$ and any $i \geq 0$ 
we have
\[
H^i((V,N),\nu^*(u_{X,*} \mathbb{L}_n)) = \varinjlim_i H^i((N_i^{\circ}),\mathbb{L}_n).
\]
If no confusion shall arise, we will
still denote $u_{X,*} \mathbb{L}_n$ (resp.~$\nu^*(u_{X,*} \mathbb{L}_n)$) by $\mathbb{L}_n$. 
\end{remark}

Recall the notion of lisse $\mathbb{Z}_p$-sheaf as in~\cite[Definition 8.1]{Sch13}.
Analogously, we make the following definition.

\begin{definition}
Let $\hat{\mathbb{Z}}_p \coloneqq \varprojlim \mathbb{Z}/p^n$ as sheaves on $X_{\mathrm{prolog}}$. Then a
\emph{lisse $\hat{\mathbb{Z}}_p$-sheaf} on $X_{\mathrm{prolog}}$ is a sheaf $\mathbb{L}$ of 
$\hat{\mathbb{Z}}_p$-modules on $X_{\mathrm{prolog}}$, such that locally $\mathbb{L}$ is isomorphic to
$\hat{\mathbb{Z}}_p \otimes_{\mathbb{Z}_p} M$, where $M$ is a finitely generated $\mathbb{Z}_p$-module.
\end{definition}

In concrete terms, $\mathbb{L}$ is a lisse $\hat{\mathbb{Z}}_p$-sheaf just means that there is a covering
$\coprod_j (V,N)_j \to (X,X)$ in $X_{\mathrm{prolog}}$ such that for each $j$ there is a finitely generated
$\mathbb{Z}_p$-module $M_j$ and a (non-canonical)-isomorphism 
\[
\mathbb{L}|_{(V,N)_j} \simeq (\hat{\mathbb{Z}}_p \otimes_{\mathbb{Z}_p} M_j)|_{(V,N)_j} \\
\coloneqq \varprojlim_m (\nu^*(u_{X,*}M_j/p^m))|_{(V,N)_j}.
\]
Note that if $X$ is connected, then all the $M_j$'s are automatically isomorphic to each other as finitely generated
$\mathbb{Z}_p$-modules.

\begin{proposition}
\label{lisse sheaf on prolog}
Let $\mathbb{L}_{\bullet}$ be a lisse $\mathbb{Z}_p$-sheaf on $U_{\mathrm{\acute{e}t}}$. 
Then $\mathbb{L} = \varprojlim \nu^*(u_{X,*}\mathbb{L}_m)$
is a lisse sheaf of $\hat{\mathbb{Z}}_p$-modules on
$X_{\mathrm{prolog}}$. This functor gives an equivalence of categories. Moreover, 
$R^j \varprojlim \nu^*(u_{X,*}\mathbb{L}_m) = 0$ for $j > 0$.
\end{proposition}

\begin{proof}
Without loss of generality, let us assume that $X$ is connected. First notice that there exists a system of finite
\'{e}tale covers $\{U_m\}$ to $U$ and a compatible system of isomorphisms $\mathbb{L}_m|_{U_m} \xrightarrow{\simeq}
(M/p^m)|_{U_m}$ where $M$ is a finitely generated $\mathbb{Z}_p$-module. By~\cite[Theorem 1.6]{Hansen17}, each $U_m$
extends to an $N_m \to X$. Let $N = \varprojlim N_m$, then $(X,N) \to (X,X)$ is a covering in $X_{\mathrm{prolog}}$.
We see that, with $\mathbb{L}$ as defined in this proposition, we have an
isomorphism $\mathbb{L}|_{(X,N)} \simeq (\hat{\mathbb{Z}}_p \otimes_{\mathbb{Z}_p} M)|_{(X,N)}$. Hence $\mathbb{L}$
as defined in this proposition is a lisse sheaf of $\hat{\mathbb{Z}}_p$-modules on
$X_{\mathrm{prolog}}$. Conversely, let $\mathbb{L}$ be a lisse sheaf of $\hat{\mathbb{Z}}_p$-modules. 
Let $(V,N)_j$ and $M_j = M$ be as in the discussion before this proposition. Then we see that 
$\mathbb{L}|_{N_j^{\circ}} \simeq (\hat{\mathbb{Z}}_p \otimes_{\mathbb{Z}_p} M_j)|_{N_j^{\circ}}$ gives rise to
a lisse sheaf of $\hat{\mathbb{Z}}_p$-modules on $U_{\mathrm{pro\acute{e}t}}$, here each $N_j = \varprojlim N_{j,l}$
is a pro-object in $V_{\mathrm{prof\acute{e}t}}$ and $N_j^{\circ} \coloneqq \varprojlim N_{j,l}^{\circ}$ naturally
is an object in $U_{\mathrm{pro\acute{e}t}}$. Therefore by~\cite[Proposition 8.2]{Sch13}, we get back a lisse
$\mathbb{Z}_p$-sheaf on $U_{\mathrm{\acute{e}t}}$. One verifies that this construction is an inverse to the functor
described in this proposition, therefore the two categories are equivalent under 
$\mathbb{L}_{\bullet} \mapsto \mathbb{L} = \varprojlim \nu^*(u_{X,*}\mathbb{L}_n)$.

To check that $R^j \varprojlim \nu^*(u_{X,*}\mathbb{L}_n) = 0$, we verify the conditions in~\cite[Lemma 3.18]{Sch13}
for $\mathcal{F}_m = \mathbb{L}_m$. 
The condition (i) of \cite[Lemma 3.18]{Sch13} trivially follows from 
the fact that $\mathbb{L}_m$ takes value in finite abelian groups.
The condition (ii) of \cite[Lemma 3.18]{Sch13} 
follows from \Cref{corokill}, \Cref{comparison logtopos}, \Cref{lemmacohomology}(1) 
and~\cite[Theorem 1.18]{Kiehl67}. Indeed, 
\cite[Theorem 1.18]{Kiehl67} tells us that there is an open cover $\{V_i\}$ of $X$ with each $V_i$ of the form 
$S \times \mathbb{D}^{r}$ (and $D \cap V_i = S \times \Delta$) as in \Cref{corokill}. Now we take $N_i$ to be 
the pro-system of all $N_{i,l}$ over $V_i$. By \Cref{comparison logtopos} and \Cref{lemmacohomology}(1), we have 
$H^j((V_i,N_i),\mathbb{L}_m) = \displaystyle \varinjlim_l H^j(N_{i,l}^{\circ}, \mathbb{L}_m)$
which is zero by \Cref{corokill}.
\end{proof}

In this note, we will only consider the case where $\mathbb{L}_m = \mathbb{Z}/p^m$.

\begin{theorem}
\label{second comparison}
We have a natural $\mathrm{Gal}(\bar{k}/k)$-equivariant isomorphism
\[
H^i_{\mathrm{\acute{e}t}}(U_{\bar{k}},\mathbb{Z}_p) 
\otimes_{\mathbb{Z}_p} B_{\mathrm{dR}}^+ \cong
H^i((X,X_{\bar{k}}),\mathbb{B}_{\mathrm{dR}}^+).
\]
\end{theorem}

\begin{remark}
\label{remark second comparison}
Here by $H^i_{\mathrm{\acute{e}t}}(U_{\bar{k}},\mathbb{Z}_p)$ we mean 
$\varprojlim_m H^i_{\mathrm{\acute{e}t}}(U_{\bar{k}},\mathbb{Z}/p^m)$. Note that $U$ is the complement of an SSNC divisor in a proper smooth adic spaces. It is easy to check that $H^i_{\mathrm{\acute{e}t}}(U_{\bar{k}},\mathbb{Z}/p^m) = 
\varinjlim_{l/k} H^i_{\mathrm{\acute{e}t}}(U_{l},\mathbb{Z}/p^m)$ 
where the left hand side is understood as the \'{e}tale cohomology of $\mathbb{Z}/p^m$ on the adic space $U_{\bar{k}}$ and the colimit on the right hand side is taking over finite field extensions $l$ of $k$.

It follows from Theorem \ref{comparison logtopos} that \[H^i(U_l,\mathbb{Z}/p^m)=H^i((X,X_l),\mathbb{Z}/p^m).\]
By \Cref{limit remark}, we can take the colimit over finite field extensions $l$ of $k$ and get
\[H^i(U_{\overline{k}},\mathbb{Z}/p^m)=H^i((X,X_{\overline{k}}),\mathbb{Z}/p^m).\]
Taking inverse limit over $m$, we have
\[\varprojlim_m H^i(U_{\overline{k}},\mathbb{Z}/p^m)=R\varprojlim_m H^i((X,X_{\overline{k}}),\mathbb{Z}/p^m)=H^i((X,X_{\overline{k}}),\varprojlim_m \mathbb{Z}/p^m)  \]
where the first identity is due to the finiteness of $H^i((X,X_{\overline{k}}),\mathbb{Z}/p^m)$ (\Cref{summary remark}) and the second identity is due to \Cref{lisse sheaf on prolog} and the fact that $R\varprojlim$ and $R\Gamma((X,X_{\overline{k}}),-)$ commutes. Therefore, we have
\[
H^i_{\mathrm{\acute{e}t}}(U_{\bar{k}},\mathbb{Z}_p) 
\cong H^i((X,X_{\bar{k}}),\hat{\mathbb{Z}}_p).
\]
\end{remark}

Before we start the proof of \Cref{second comparison}, we need a preliminary discussion
on A-R $p$-adic projective systems, c.f.~\cite[10.1]{fulei}.

\begin{lemma}
\label{lemma take inverse limit}
Let $\mathbb{L} = \varprojlim \nu^*(u_{X,*}\mathbb{L}_m)$
be a lisse sheaf of $\hat{\mathbb{Z}}_p$-modules on
$X_{\mathrm{prolog}}$. 
Let $H_m$ be the cohomology group
$H^i(U_{\overline{k}},\mathbb{L}_m)=H^i((X,X_{\overline{k}}),\mathbb{L}_m)$. 
Then the system $(H_m)_{m\in \mathbb{N}}$ is A-R $p$-adic.
\end{lemma}

\begin{proof}
The proof is similar to the case of schemes. We may assume that the inverse system $\mathbb{L}^{\bullet}$ satisfies $\mathbb{L}_{m+1}/p^m\cong \mathbb{L}_{m}$. We apply results in the theory of $l$-adic systems to prove this lemma. In fact, we denote $R\Gamma (U_{\overline{k}},\mathbb{L}_m)$ by $K^{\bullet}_m$. We claim that the natural maps 
\[
\label{un}
\tag{\epsdice{4}}
u_n: K_{n+1}^{\bullet} \otimes^{L}_{\mathbb{Z}/p^{n+1}} \mathbb{Z}/p^n \xrightarrow{\cong} K_n^{\bullet} 
\]
are isomorphisms in the derived category. Note that $H^j(K^{\bullet}_n)=H^j(U_{\overline{k}},\mathbb{L}_n)$ is zero if $j\notin [0,2\dim (X)]$.
Represent each $K^{\bullet}_n$ by a bounded above complex of flat $\mathbb{Z}/p^n$-modules with $K_n^j=0$ for $j> 2\dim(X)$. Moreover, the complex $\ldots\rightarrow K_n^{-1}\rightarrow K_n^0 \rightarrow 0$ is a resolution of $\mathrm{coker}(K_n^{-1}\rightarrow K_n^0)$ by flat $\mathbb{Z}/p^{n}$-modules. It follows that
\[\mathrm{Tor}^i_{\mathbb{Z}/p^n}\bigg( (\mathrm{coker}(K_n^{-1}\rightarrow K_n^0), \mathbb{Z}/p\bigg)= H^{-i}( K_n^{\bullet}\otimes_{\mathbb{Z}/p^n} \mathbb{Z}/p)= H^{-i}(K_1^{\bullet})=0\]
for $i>0$ where we use the fact that $K_n^{\bullet} \otimes_{\mathbb{Z}/p^n} \mathbb{Z}/p \cong K_1^{\bullet}.$ Therefore, by the local flatness criterion \cite[Theorem 22.3]{Mat}, we conclude that $\mathrm{coker}(K_n^{-1}\rightarrow K_n^0)$ is a flat $\mathbb{Z}/p^n$-modules. It follows that the complex $K_n^{\bullet}$ is quasi-isomorphic to the bounded complex of flat $\mathbb{Z}/p^n$-modules
\[0\rightarrow \mathrm{coker}(K_n^{-1}\rightarrow K_n^0)\rightarrow K_n^1\rightarrow\ldots \rightarrow K_n^{2\dim (X)}\rightarrow 0 .\]
By \cite[Lemma 10.1.14]{fulei}, each complex $K_n^{\bullet}$ is isomorphic in the derived category to a complex $L_n^{\bullet}$ of free $\mathbb{Z}/p^n$-modules of finite ranks with $L^{j}_n=0$ for $j\notin [0,2\dim(X)]$. The natural isomorphism $u_n$ gives an isomorphism \[v_n:L_{n+1}^{\bullet} \otimes^{L}_{\mathbb{Z}/p^{n+1}} \mathbb{Z}/p^n \xrightarrow{\cong} L_n^{\bullet}\] in the derived category. By \cite[Lemma 10.1.13]{fulei}, this isomorphism $v_n$ is induced by a quasi-isomorphism $L_{n+1}^{\bullet} \otimes^{L}_{\mathbb{Z}/p^{n+1}} \mathbb{Z}/p^n \xrightarrow{\cong} L_n^{\bullet}$ of complexes. We apply \cite[Proposition 10.1.15]{fulei} to the system $(L^{\bullet}_m)_{m\in \mathbb{Z}}$ and show that $H^i(L_m^{\bullet})=H_m$ is A-R $p$-adic.

We give a proof of our claim as follows.
\begin{lemma}
The natural morphism $u_n$ (see \ref{un}) is an isomorphism in the derived category.
\end{lemma}

\begin{proof}
Take an injective resolution of the $\mathbb{Z}/p^m$-modules 
\[\mathbb{L}_m\xrightarrow{qis} I^0\rightarrow I^1\rightarrow \ldots.\] 
Note that $H^j(K^{\bullet}_m)=H^j(U_{\overline{k}},\mathbb{L}_m)$ is zero if $j\notin [0,2\dim (X)]$. The truncated complex $I'^{\bullet}$
\[I^0\rightarrow \ldots \rightarrow \mathrm{Im}(I^{2\dim(X) -1} \rightarrow I^{2\dim(X)})\rightarrow 0\]
is an $R\Gamma(U_{\overline{k}},-)$-acyclic resolution of $\mathbb{L}_m$. In the following, we let $m=n+1$.
Take a resolution $A^{\bullet}$ of $\mathbb{Z}/p^n$ by free $\mathbb{Z}/p^{n+1}$-modules
\[\ldots\rightarrow A^{-1} \rightarrow A^0 \rightarrow \mathbb{Z}/p^{n} \rightarrow 0.\]
We have that
\[\begin{aligned}
\mathbb{Z}/p^{n} \otimes^L_{\mathbb{Z}/p^{n+1}} K^{\bullet}_{n+1} &\cong A^{\bullet} \otimes_{\mathbb{Z}/p^{n+1}} \Gamma(U_{\overline{k}}, I'^{\bullet})\\
&\cong \Gamma(U_{\overline{k}}, A^{\bullet}\otimes_{\mathbb{Z}/p^{n+1}} I'^{\bullet})\\
&\cong R\Gamma(U_{\overline{k}}, \mathbb{Z}/p^n \otimes^L_{\mathbb{Z}/p^{n+1}} I'^{\bullet})\\
&\cong R\Gamma(U_{\overline{k}}, \mathbb{L}_n)=K_n^{\bullet}
\end{aligned}
 \]
	where the second isomorphism is due to that $A^i$ are free $\mathbb{Z}/p^{n+1}$-modules, the third isomorphism is due to that $A^i\otimes I'^j$ is $R\Gamma$-acyclic and the last isomorphism is due to our assumption $\mathbb{L}_{n+1}/{p^n}\cong \mathbb{L}_{n} $  .
\end{proof}

\end{proof}

\begin{proof}[Proof of \Cref{second comparison}]
This follows from the argument in~\cite[Theorem 8.4]{Sch13}, for the sake of completeness let us repeat the argument
in below.

First we claim that
\[
H^i((X,X_{\bar{k}}),\mathbb{Z}/p^m) \otimes_{\mathbb{Z}_p} A_{\mathrm{inf}}^a \cong
H^i((X,X_{\bar{k}}), \mathbb{A}_{\mathrm{inf}}^a/p^m).
\]
Indeed, when $m=1$ this follows from \Cref{rmkpc} 
(applied to $\mathbb{L} = \mathbb{F}_p$) and the general case follows from induction. 
Notice that the almost setting here is with respect to 
$[\hat{\mathfrak{m}}]$, the ideal generated by $([a],a \in \hat{\mathfrak{m}})$
where $\hat{\mathfrak{m}}$ is the maximal ideal in $\hat{\bar{k}}^{\circ}$. Now the sheaves 
$\mathbb{A}_{\mathrm{inf}}^a/p^m$ satisfy the hypotheses of the almost
version of~\cite[Lemma 3.18]{Sch13}. Therefore we may pass to the inverse limit 
$\mathbb{A}_{\mathrm{inf}}^a$ and get an almost isomorphism
\[
H^i((X,X_{\bar{k}}),\hat{\mathbb{Z}}_p) \otimes_{\mathbb{Z}_p} A_{\mathrm{inf}}^a \cong
H^i((X,X_{\bar{k}}),  \mathbb{A}_{\mathrm{inf}}^a).
\]
Now we invert $p$ and get almost isomorphisms
\[
H^i((X,X_{\bar{k}}),\hat{\mathbb{Z}}_p) \otimes_{\mathbb{Z}_p} B_{\mathrm{inf}}^a \cong
H^i((X,X_{\bar{k}}), \mathbb{B}_{\mathrm{inf}}^a).
\]
Since $[\hat{\mathfrak{m}}]$ becomes the unit ideal in $B_{\inf}/\ker(\theta)$, multiplication by $\xi^l$ 
(where $\xi$ is any generator in $\ker(\theta)$) gives that
\[
H^i((X,X_{\bar{k}}),\hat{\mathbb{Z}}_p) \otimes_{\mathbb{Z}_p} B_{\mathrm{inf}}/(\ker(\theta))^l \cong
H^i((X,X_{\bar{k}}), \mathbb{B}_{\mathrm{inf}}/(\ker(\theta))^l).
\]
Again the sheaves $\mathbb{B}_{\mathrm{inf}}/(\ker(\theta))^l$
satisfy the conditions in~\cite[Lemma 3.18]{Sch13}, hence we have that
\[
H^i((X,X_{\bar{k}}),\hat{\mathbb{Z}}_p) \otimes_{\mathbb{Z}_p} B_{\mathrm{dR}}^+ \cong
H^i((X,X_{\bar{k}}),\mathbb{B}_{\mathrm{dR}}^+),
\]
which is what we want by \Cref{remark second comparison}.
\end{proof}

Finally let us show \Cref{Main Theorem}, which we restate below.

\begin{theorem}
The Hodge--de Rham spectral sequence
\[
\xymatrix{ 
E_1^{j,i} = H^i(X,\Omega_X^j(\log D))\ar@{=>}[r] & H^{i+j}(X,\Omega^{\bullet}_{X}(\log D))
}
\]
degenerates, and there is a $\mathrm{Gal}(\bar{k}/k)$-equivariant isomorphism
\[
H^i_{\mathrm{\acute{e}t}}(U_{\bar{k}},\mathbb{Z}_p) 
\otimes_{\mathbb{Z}_p} B_{\mathrm{dR}} \cong
H^i(X,\Omega^{\bullet}_{X}(\log D)) \otimes_k B_{\mathrm{dR}}
\]
preserving filtrations. In particular, there is also a $\mathrm{Gal}(\bar{k}/k)$-equivariant isomorphism
\[
H^i_{\mathrm{\acute{e}t}}(U_{\bar{k}},\mathbb{Z}_p) 
\otimes_{\mathbb{Z}_p} \hat{\bar{k}} \cong
\bigoplus_j H^{m-j}(X,\Omega_X^j(\log D ))\otimes_k \hat{\bar{k}}(-j).
\]
\end{theorem}

\begin{proof}
By~\Cref{second comparison}, we have
\[
H^i_{\mathrm{\acute{e}t}}(U_{\bar{k}},\mathbb{Z}_p) 
\otimes_{\mathbb{Z}_p} B_{\mathrm{dR}}^+ \cong
H^i((X,X_{\bar{k}}),\mathbb{B}_{\mathrm{dR}}^+).
\]
In particular, $H^i((X,X_{\bar{k}}),\mathbb{B}_{\mathrm{dR}}^+)$ 
is a free $B_{\mathrm{dR}}^+$-module of finite rank.
This, together with~\Cref{first comparison}, implies that
\[
\sum_j \dim_k H^{i-j}(X,\Omega_X^j(\log D)) = 
\dim_{B_{\mathrm{dR}}} 
(H^{i}(X,\Omega^{\bullet}_{X}(\log D)) \otimes_k B_{\mathrm{dR}}),
\]
hence the Hodge--de Rham spectral sequence degenerates. Also by~\Cref{first comparison}, we get
\[
H^i_{\mathrm{\acute{e}t}}(U_{\bar{k}},\mathbb{Z}_p) 
\otimes_{\mathbb{Z}_p} B_{\mathrm{dR}} \cong
H^i((X,X_{\bar{k}}),\mathbb{B}_{\mathrm{dR}}^+) \otimes_{B_{\mathrm{dR}}^+} B_{\mathrm{dR}} \cong H^i(X,\Omega^{\bullet}_{X}(\log D)) \otimes_k B_{\mathrm{dR}}.
\]
\end{proof}

\bibliographystyle{amsalpha}
\bibliography{padichodge}

\end{document}